%% file: StratificationSensitivity.tex
\documentclass[final]{siamltex}

\usepackage{hyperref} %BUG if put after
\usepackage{enumitem}
\usepackage{rotating} % for side captions

\pdfoutput=1

\usepackage{mystyle}

\graphicspath{{./images/}}

\hypersetup{  
   bookmarks=true,
   backref=true,
   pagebackref=false,
   colorlinks=true,
   linkcolor=blue,
   citecolor=red,
   urlcolor=blue,
   pdftitle={Sensitivity Analysis for Mirror-Stratifiable Convex Functions},
   pdfauthor={Jalal Fadili, J\'er\^ome Malick, Gabriel Peyr\'e},
   pdfsubject={Sensitivity Analysis for Mirror-Stratifiable Convex Functions}
}

\newcommand{\Param}{\Pi}
\newcommand{\Dsol}{\mathfrak{D}}

\title{Sensitivity Analysis\\for Mirror-Stratifiable Convex Functions}

\author{Jalal Fadili\thanks{Normandie Univ, ENSICAEN, CNRS, GREYC, France.} \and J\'er\^ome Malick\thanks{CNRS and LJK, Grenoble, France} \and Gabriel Peyr\'e\thanks{CNRS and DMA, ENS Paris, France.} }

\begin{document}

\maketitle

\input{sections/abstract}

% REQUIRED
\begin{keywords}
convex analysis, inverse problems, sensitivity, active sets, first-order splitting algorithms, applications in imaging and machine learning
\end{keywords}

% REQUIRED
\begin{AMS}
65K05, 65K10, 90C25, 90C31.
\end{AMS}

%\tableofcontents

\input{sections/intro}
\input{sections/stratif-functions}

% \input{sections/applications}
%
\input{sections/composite-sensitivity}

\input{sections/inverse-problems}

\input{sections/algorithms}
%
\input{sections/numerics}
\input{sections/conclusion}

\appendix

\input{sections/appendix}

\bibliographystyle{plain}
\bibliography{biblio-algo,biblio-imaging,biblio-sensitivity}

\end{document}

%% file: sections/abstract.tex
% !TEX root = ../StratificationSensitivity.tex

\begin{abstract}
This paper provides a set of sensitivity analysis and activity identification results for a class of convex functions with a strong geometric structure, that we coined ``mirror-stratifiable''.
These functions are such that there is a bijection between a primal and a dual stratification of the space into partitioning sets, called strata.  
This pairing is crucial to track the strata that are identifiable by solutions of parametrized optimization problems or by iterates of optimization algorithms. 
This class of functions encompasses all regularizers routinely used in signal and image processing, machine learning, and statistics. %Popular examples are the $\ell_1$ norm, the nuclear norm and the total variation semi-norm. %and their numerous generalizations.  
We show that this ``mirror-stratifiable'' structure enjoys a nice sensitivity theory, allowing us to study stability of solutions of optimization problems to small perturbations, as well as activity identification of first-order proximal splitting-type algorithms.
Existing results in the literature typically assume that, under a non-degeneracy condition, the active set associated to a minimizer is stable to small perturbations and is identified in finite time by optimization schemes. In contrast, our results do not require any non-degeneracy assumption: in consequence, the optimal active set is not necessarily stable anymore, but we are able to track precisely the set of identifiable strata.% Under non-degeneracy, the active strata reduce to a single one.
We show that these results have crucial implications when solving challenging ill-posed inverse problems via regularization, a typical scenario where the non-degeneracy condition is not fulfilled. Our theoretical results, illustrated by numerical simulations, allow us to characterize the instability behaviour of the regularized solutions, by locating the set of all low-dimensional strata that can be potentially identified by these solutions.
\end{abstract}

%% file: sections/intro.tex
% !TEX root = ../StratificationSensitivity.tex

%%%%%%%%%%%%%%%%%%%%%%%%%%%%%%%%%%%%%%%%%%%%%%%%%%%%%%%%%%%%
%%%%%%%%%%%%%%%%%%%%%%%%%%%%%%%%%%%%%%%%%%%%%%%%%%%%%%%%%%%%
\section{Introduction}
%%%%%%%%%%%%%%%%%%%%%%%%%%%%%%%%%%%%%%%%%%%%%%%%%%%%%%%%%%%%
%%%%%%%%%%%%%%%%%%%%%%%%%%%%%%%%%%%%%%%%%%%%%%%%%%%%%%%%%%%%

Variational methods and non-smooth optimization algorithms are ubiquitous to solve large-scale inverse problems in various fields of science and engineering, and in particular in data science. The non-smooth structure of the optimization problems promotes solutions conforming to some notion of simplicity or low-complexity (e.g.\;sparsity, low-rank, etc.). The low-complexity structure is often manifested in the form of a low-dimensional ``active set". It is thus of prominent interest to be able to quantitatively characterize the stability of these active sets to perturbations of the objective function. Of crucial importance is also the identification in finite time of these active sets by iterates of optimization algorithms that numerically minimize the objective function. This type of problems and results is referred to as ``activity identification''; a review of the relevant literature will be provided in the sequel (at the beginning of each section). 
In a nutshell, the existing identification results guarantee a perfect stability of \hl{the active set} to perturbations, or a finite identification of \hl{this active set} via algorithmic schemes, under some non-degeneracy conditions (see in particular \cite{lewis2002active,drusvyatskiy2013optimality,2014-vaiter-ps-review}). The crucial non-degeneracy assumption, that takes generally the form~\eqref{eq-ri-condition}, can be viewed as a geometric generalization of strict complementary in non-linear programming. However, as we will illustrate shortly through some preliminary numerics, such a condition is too stringent and is often barely verified. The goal of this paper is to investigate the situation where this non-degeneracy assumption is violated. 
%In summary, our main finding is that, under a suitable structural assumption on the objective function, that we coin ``mirror-stratifiable'', it is still possible to precisely locate the solutions in an ordered family of active sets. 

%%%%%%%%%%%%%%%%%%%%%%%%%%%%%%%%%%%%%%%%%%%%%%%%%%%%%%%%%%%%%%%%%%%%%%%%%%%%%%%
\subsection{Motivating Examples}
\label{sec-motivation}

In order to better grasp the relevance of our analysis, let us first start with the setting of inverse problems that pervades various fields including signal processing and machine learning.
We will come back to this setting in Section~\ref{sec-ip-regul} with further discussions and references.

%%%
\paragraph{Regularized Inverse Problems}

Assume one observes 
\eql{\label{eq-fwd-ip}
	y = \yorig + w \in \RR^P
	\qwhereq
	\yorig \eqdef \Phi \xorig
}
where $w$ is some perturbation (called noise) and $\Phi \in \RR^{P \times N}$ (called forward operator, or design matrix in statistics). Solving an inverse problem amounts to recovering $\xorig$, to a good approximation, knowing $y$ and $\Phi$ according to \eqref{eq-fwd-ip}. Unfortunately, in general, $P$ can be much smaller than the ambient dimension $N$, and when $P = N$, the mapping $\Phi$ is in general ill-conditioned or even singular. 
%This entails that the inverse problem is in general ill-posed. In order to reach the land of well-posedness, it is necessary to restrict the set of candidate solutions to those resembling $\xorig$.

A classical approach is then to assume that $\xorig$ has some "low-complexity", and to use a prior regularization $R$ promoting solutions with such low-complexity.
This leads to the following optimization problem%\footnote{Without loss of generality, we assume that $y$ is in the range of $\Phi$ for the linear constraint in \eqref{eq-variational-ip-intro-noiseless} to be feasible.}
\begin{align}\label{eq-variational-ip-intro-noisy}\tag{$\Pp(\la,y)$}
	&\umin{x \in \RR^N}  E(x,(\la,y)) \eqdef R(x) + \frac{1}{2\la}\norm{ y - \Phi x }^2, 
	%&\text{for } \la>0
%	\label{eq-variational-ip-intro-noiseless}\tag{$\Pp(0,y)$}
%	&\umin{x \in \RR^N} E(x,(0,y)) \eqdef R(x) \quad\text{s.t.}\quad \Phi x  =  y, 
%	&\text{for } \la=0.
\end{align}
%The function plays the role of the regularization that promotes solutions of \eqref{eq-variational-ip-intro-noisy} (assuming they exist) bearing features similar to those of $\xorig$. 
%, while in the general case, one solves~\eqref{eq-variational-ip-intro-noisy} with $\la>0$ that is typically chosen proportionally to the noise level $\norm{w}=\norm{y-\yorig}$.
%%%
%\paragraph{Sparse, low-rank and low-complexity models}
Let us discuss two popular examples of regularizing functions promoting a low-complexity structure. These two functions will be used in our numerical experiments.

\begin{exmp}[$\ell_1$ norm]\label{ex:l1}
		For $x \in \RR^N$, its $\ell_1$ norm reads
		\eql{\label{eq-lun-norm}
			R(x) = \norm{x}_1 \eqdef \sum_{i=1}^N |x^i|,
		}
where $x^i$ is the $i$-th entry of $x$. As advocated for instance by~\cite{chen1999atomi,tibshirani1996regre}, the $\ell_1$ enforces the solutions of~\eqref{eq-variational-ip-intro-noisy} to be \textit{sparse}, i.e. to have a small number of non-zero components. Indeed, the $\ell_1$ norm can be shown to be the tightest convex relaxation (in the sense of bi-conjugation) of the $\ell_0$ pseudo-norm restricted to the unit Euclidian ball~\hl{\cite{leconvrank12}}. Recall that $\ell_0$ pseudo-norm of $x \in \RR^N$ measures its sparsity  
		\eq{
			\norm{x}_0 \eqdef \sharp\enscond{i \in \{1,\ldots,N\}}{x^i \neq 0}.
		}
		Sparsity has witnessed a huge surge of interest in the last decades. 
		For instance, in signal and imaging sciences, one can approximate most natural signals and images using sparse expansions in an appropriate dictionary (see e.g.~\cite{mallat2009a-wav}). 
		In statistics, sparsity is a key toward model selection and interpretability%of the results for supervised classification or regression tasks
~\cite{buhlmann2011statistics}. 
\end{exmp}

\begin{exmp}[Nuclear norm]\label{ex:nuc} 
		For a matrix $x \in  \RR^{n_1 \times n_2} \sim \RR^N$, where $N=n_1n_2$, the nuclear norm is defined as 
		\eq{
			\norm{x}_* \eqdef \norm{\si(x)}_1 = \sum_{i=1}^{n} \si^i(x), 
		}
		where $n \eqdef \min(n_1,n_2)$ and 
	 	$\si(x) = (\si^1(x),\ldots,\si^n(x)) \in \RR_+^n$ is the vector of singular values of $x$. 
		The nuclear norm is the tightest convex relaxation of the rank (in the sense of bi-conjugation) restricted to the unit Frobenius ball~\cite{hiriart2012convex}. This underlies its wide use to promote solutions of~\eqref{eq-variational-ip-intro-noisy} with low rank, where we recall $\rank(x) \eqdef \norm{\si(x)}_0$. Low-rank regularization has proved useful for a variety of applications, including control theory and machine learning; see~e.g.\;\cite{fazel2002matrix,bach2008trace,candes2009exact}%,recht2010guaranteed,CandesRPCA11,CandesLowRank13}.
\end{exmp}

%In a similar spirit to the two examples, a wealth of non-smooth regularizing functions have been designed in the literature to capture low-complexity models. One can think of the $\ell_{1,2}$-norm (to promote group sparse vectors)~\cite{yuan2005model}, the total variation seminorm~\cite{rudin1992nonlinear} (to favour piecewise-constant vectors).
%%, and many others, including those that can be built via linear combinations of the above regularizers (e.g. $\norm{\cdot}_1+\norm{\cdot}_*$ as proposed in~\cite{SavalleRV12,Golbabaee12,Otazo15}). We refer the reader to~\cite{2014-vaiter-ps-review}~for a detailed review. \\

%%%
\paragraph{Active set(s) identification}

The above discussed regularizers $R$ are non-smooth convex functions. This non-smoothness arises in a highly structured fashion and is usually associated, locally, with some low-dimensional active subset of $\RR^N$ (in many cases, such a subset is an affine or a smooth manifold). Thus $R$ will favor solutions of~\eqref{eq-variational-ip-intro-noisy} that lie in a low-dimensional active set \hl{and would allow the inversion of} the system~\eqref{eq-fwd-ip} in a \textit{stable} way. More precisely, one would like that under small perturbations $w$, the solutions of~\eqref{eq-variational-ip-intro-noisy} move stably along the active set. A byproduct of this behaviour is that from an algorithmic perspective, if an optimization algorithm is used to solve~\eqref{eq-variational-ip-intro-noisy}, one would hope that the iterates of the scheme identify the active set in finite time. 

Identifying the low-dimensional active set in a stable way is highly desirable for several reasons. One reason is that it is a fundamental property most practitioners are looking for. Typical examples include neurosciences~\cite{Ekanadham2014} where the goal is to recover a spike train from neural activity, or astrophysics~\cite{StarckBook06} where it is desired to separate stars from a background. In both examples, sparsity can be used as a modeling hypothesis and the recovery method should comply with it in a stable way. A second reason is algorithmic since one can also take advantage of the low-dimensionality of the identified active set to reduce computational burden and memory storage, hence opening the door to higher-order acceleration of optimization algorithms (see~\cite{Lemarechal-ULagrangian,miller2005newton}). 
% Such identification and exploitation of substructure is thus of importance in the context of parallel computing where the learning problem size artificially increases in a very structured manner.

Unfortunately, %despite the effort invested by many researchers to provide sufficient non-degeneracy-type conditions ensuring exact active set identification and stability to perturbations, most practitioners know that 
this desirable behaviour rarely occurs in practical applications. Yet one still observes some form of partial stability (to be given a rigorous meaning in Section~\ref{sec-sentivity-hybrid} and~\ref{sec-ip-regul}), as confirmed by the numerical experiment of the next paragraph. 

%%%
\subsection{Illustrative numerical experiment}
\label{sec-motivnum}

We consider a simple ``compressed sensing'' scenario, with the $\ell_1$ norm $R=\norm{\cdot}_1$~\cite{candes2005decoding} and the nuclear norm $R=\norm{\cdot}_*$~\cite{candes2010power} as regularizers. 
The operator $\Phi \in \RR^{N \times P}$ is drawn uniformly at random from the standard Gaussian ensemble, i.e. the entries of $\Phi$ are independent and identically distributed Gaussian random variables with zero-mean and unit variance.
In the case $R=\norm{\cdot}_1$, we set $(N,P)=(100,50)$ and the vector to recover $\xorig$ is drawn uniformly at random among sparse vectors with $R_0(\xorig) \eqdef \norm{\xorig}_0=10$ and unit non-zero entries. 
In the case $R=\norm{\cdot}_*$, we set $(N,P)=(400,300)$ ($n_1=n_2=n=20$) and the matrix to recover $\xorig$ is drawn uniformly at random among low-rank matrices with $R_0(\xorig) \eqdef \rank(\xorig)=4$ and unit non-zero singular values. 
For each problem suite $(R,N,P)$, $1000$ realizations of $(\xorig,\Phi,w)$ are drawn, and $y$ is then generated according to \eqref{eq-fwd-ip}.
% If $w=0$, and \eqref{eq-variational-ip-intro-noiseless} is solved, otherwise, 
%
The entries of the noise vector $w$ are drawn uniformly from a Gaussian with standard deviation $0.1$, we set $\la=0.28$ for $R=\norm{\cdot}_1$ and $\la=10$ for $R=\norm{\cdot}_*$. 

For each realization of $(\xorig,\Phi,w)$, we solve the associated problem \eqref{eq-variational-ip-intro-noisy} using CVX~\hl{\cite{cvx}} to get a high precision; denote $\xsol(\la,y)$ the obtained solutions. 
We also solve~\eqref{eq-variational-ip-intro-noisy} by the Forward-Backward (FB) scheme which reads in this case\footnote{\hl{The definition of the proximal mapping $\prox_{\tau R}$ (for $\tau > 0$) is given in~\eqref{eq:prox}.} The proximal mappings of the $\ell_1$ and nuclear norms are
\begin{align*}
	\prox_{\mu \norm{\cdot}_1}(x) = \pa{ \sign(x^i) \max(0,|x^i|-\mu) }_i, \;\;
	\prox_{\mu \norm{\cdot}_*}(x) = U \diag( \prox_{\mu \norm{\cdot}_1}(\si(x)) ) V^*, 
\end{align*}
where $x=U \diag(\si(x)) V^*$ is a SVD decomposition of $x$.}
\eq{\label{eq-fb-iter-intro}
	\Iter{x} = \prox_{\la\gamma R}(\iter{x}+\gamma \Phi^*(y - \Phi \iter{x})) .
}
%\hl{where $\prox_{\mu R}$ is the proximal mapping of $\mu R$ for $\mu>0$.}
In our setting, with  \hl{$\ga = 1.8 / \sigma_{\max}(\Phi^*\Phi)$} and non-emptiness of $\Argmin(E(\cdot,\la,y)$,  it is well-known that the sequence $(\iter{x})_{k \in \NN}$ converges to a point in $ \Argmin(E(\cdot,\la,y))$. 
%Figure~\ref{fig-sample-paths} gathers some interesting information.
%and observe its identification properties
  % and a well-chosen constant $c_0 > 0$. 
%
% pour simplifier, ne parlons pas de ca :
%
%The first observation, which is anticipated by the standard compressed sensing theory (see the review in~\cite{2014-vaiter-ps-review}) is that in the noiseless case ($w=0$), and for the simulated complexity index $R_0(\xorig)$ (sparsity $10$ for $\ell_1$ and rank $4$ for the nuclear norm), $\xorig$ is the unique solution to $(\Pp(0,\yorig))$ with high probability (with respect to $\Phi$ and $\xorig$). In turn, $\xorig$ will be uniquely recovered from $y=\yorig$ by solving $(\Pp(0,\yorig))$ with high probability.
%However a very different picture emerges when even a small perturbation noise $w$ is added to $\yorig$ (we here assumed random white Gaussian with standard deviation $10^{-3}\norm{\yorig}$, but the same conclusion remains unchanged for a noise of arbitrary small amplitude). 

The top row of Figure~\ref{fig-sample-paths} displays the histogram of the complexity index excess 
$\hl{\delta \eqdef R_0(\xsol(\la,y))-R_0(\xorig)}$
% where $\xsol(\la,y)$ is any solution to~\eqref{eq-variational-ip-intro-noisy} using $\la = c_0 \norm{w}$. 
%
%$c_0$ is chosen large enough as devised in Section~\ref{sec-ip-regul-sensitiv}, in order to achieve sufficient denoising effect (otherwise the solution does not have a low-complexity), but its precise value does not affect the observed identification results we are about to describe. 
%
%These histograms of the complexity index excess are computed 
%
%From these simulations, 
which clearly shows that we do not have exact stability under the perturbations $w$. Indeed, although $R_0(\xsol(\la,y))$ is still rather small, its value is in most cases \textit{larger} than the complexity index $R_0(\xorig)$.
The bottom row of Figure~\ref{fig-sample-paths} depicts the evolution with $k$ of the complexity index $R_0(\iter{x})$ for two random instances of $\xorig$ (in blue and red) corresponding to two different values of $\hl{\delta}$. One can observe that the iterates identify in finite time some low-dimensional active set with a complexity index strictly larger than the one associated to $\xorig$ (the red curves). 
We will provide a more detailed discussion of this phenomenon in Section~\ref{sec-numerics}.% in the light of our theoretical results of Section~\ref{sec-algorithm}.

\newcommand{\sidecap}[1]{ {\begin{sideways}\parbox{3.5cm}{\centering\small #1}\end{sideways}} }

\begin{figure}
\begin{tabular}{@{}c@{}c@{\hspace{5mm}}c@{}}
\sidecap{Histograms of $\de$} &
\includegraphics[width=.45\linewidth]{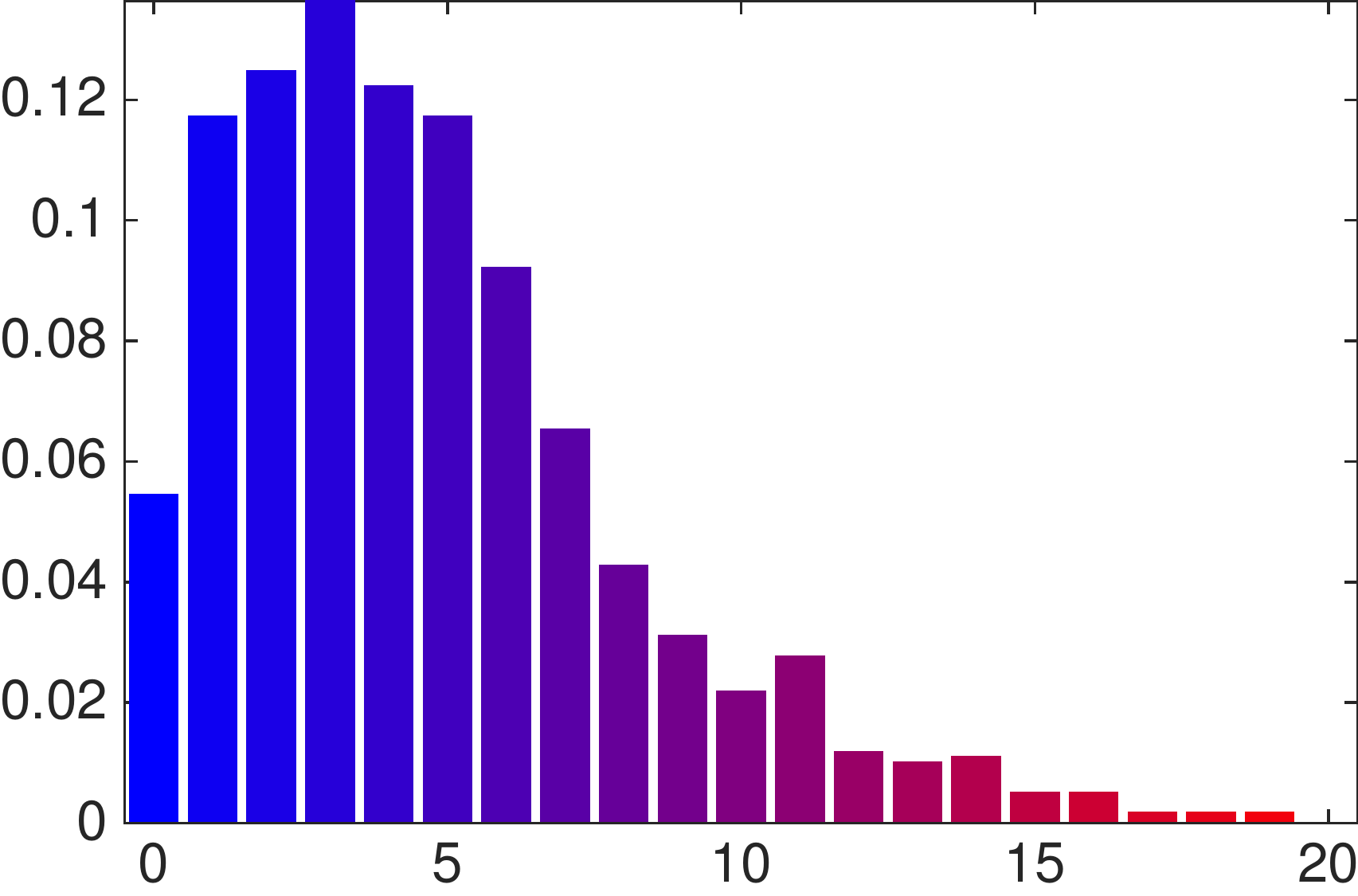} & 
\includegraphics[width=.45\linewidth]{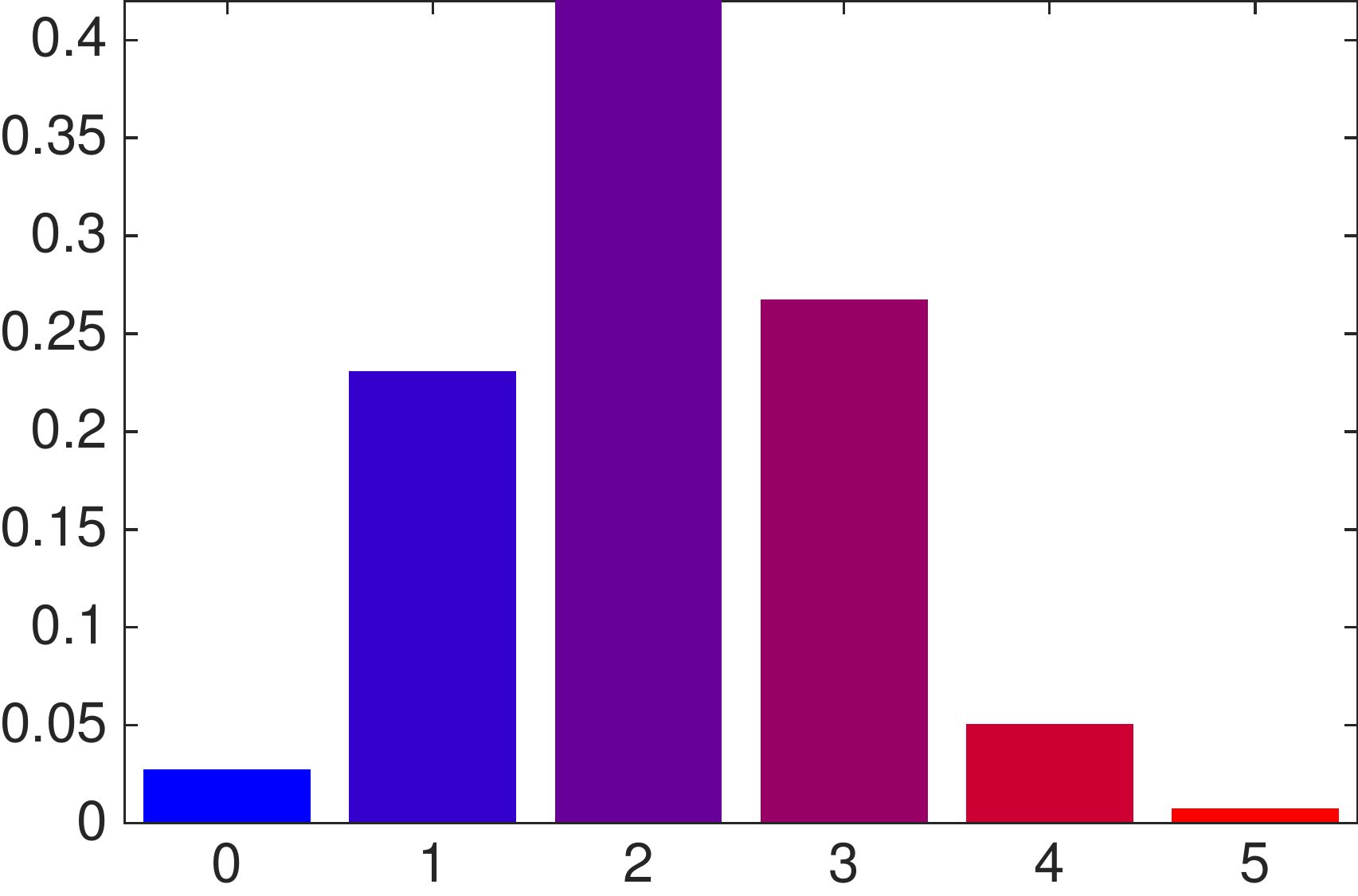} \\
\sidecap{$R_0(\iter{x})$ vs iteration $k$} &
\includegraphics[width=.45\linewidth]{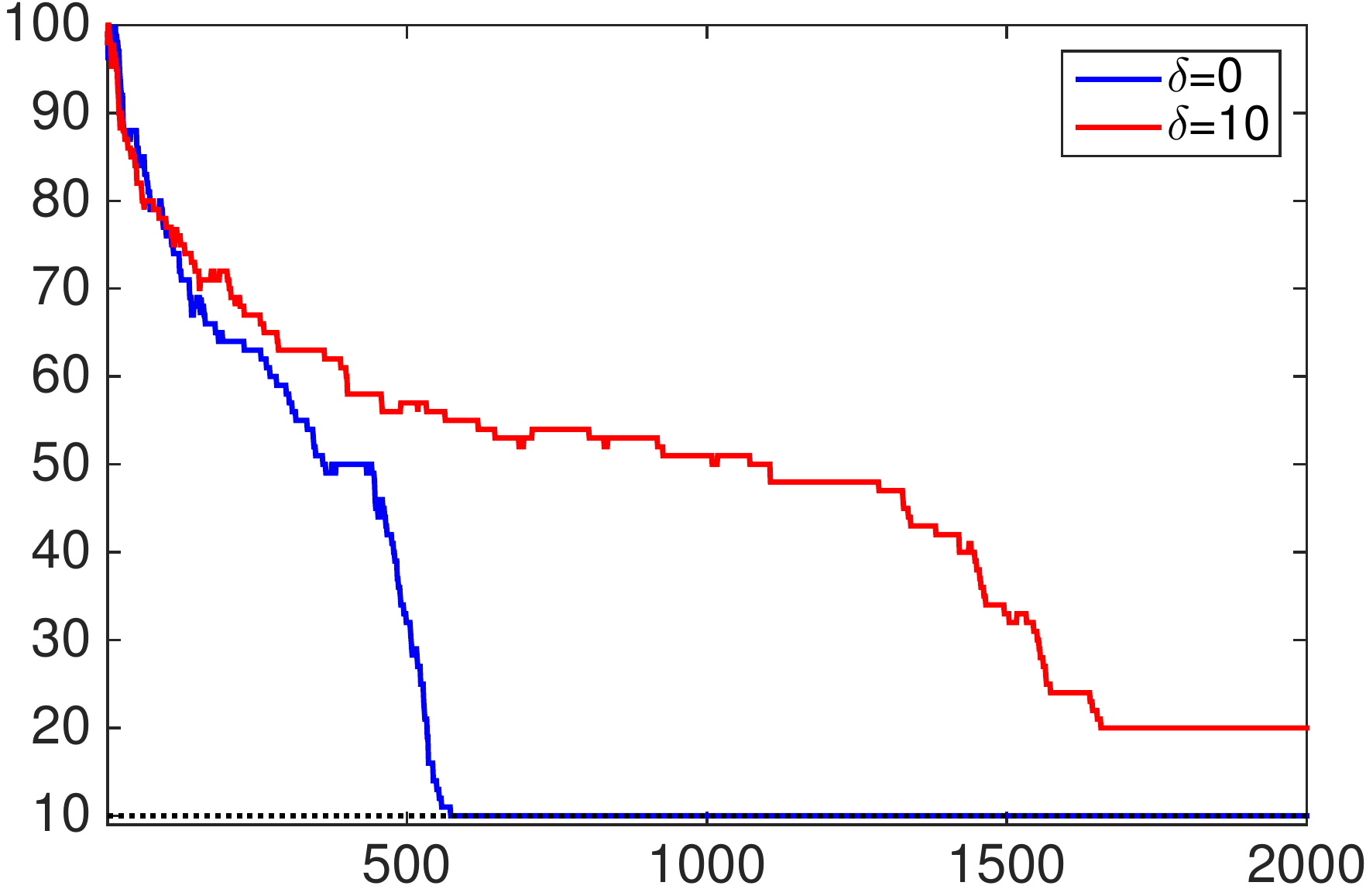} &
\includegraphics[width=.45\linewidth]{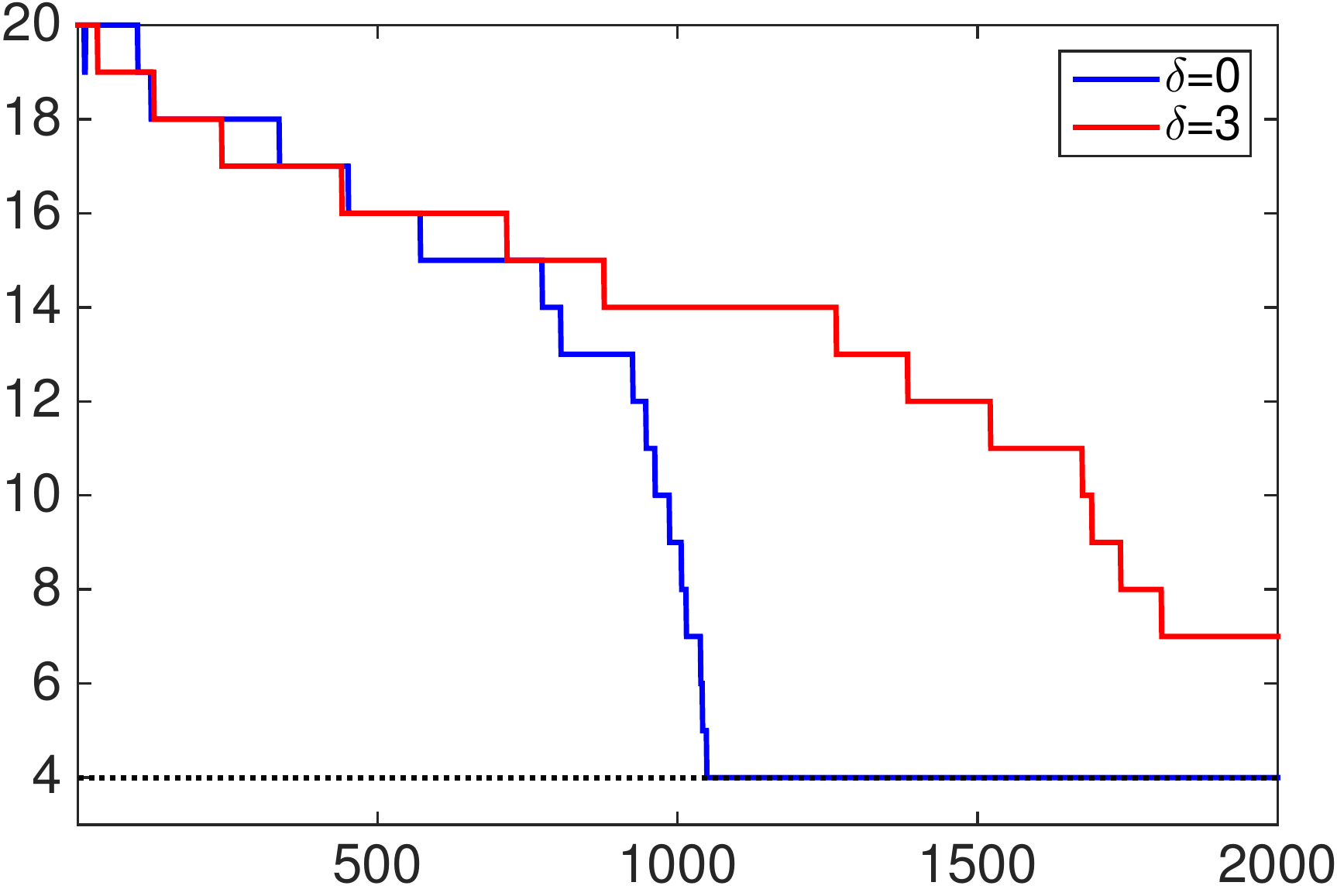}\\
& $R=\norm{\cdot}_1$ & $R=\norm{\cdot}_*$ 
\end{tabular}
\caption{\label{fig-sample-paths}
\textbf{Top row:} 
Histogram of the complexity index excess $\delta = R_0(\xsol(\la,y))-R_0(\xorig)$ where $\xsol(\la,y)$ is the solution of~\eqref{eq-variational-ip-intro-noisy}
with noisy observation $y=\yorig+w$ for a small perturbation $w$ and a well-chosen parameter value $\la = c_0 \norm{w}$.
\textbf{Bottom row:}
Evolution of the complexity index $R_0(\iter{x})$ where $(\iter{x})_{k \in \NN}$ is the FB iterates sequence converging to some $\xsol(\la,y) \in \Argmin(E(\cdot,\la,y))$.
}
\end{figure}

As we will emphasize in the review of previous work, existing results on sensitivity analysis and low-complexity regularization only focus the case where $R_0(\xsol(\la,y)) = R_0(\xorig)$ \hl{(i.e. $\delta=0$)}, whose underlying claim is that the low-complexity model is perfectly stable to perturbations, which typically requires some non-degeneracy assumption to hold.  
In many situations, however, including the compressed sensing scenario described above, but also in super-resolution imaging (see~\cite{2015-duval-thin-grids}) and other challenging inverse problems, non-degeneracy-type hypotheses are too restrictive. 
It is the goal of this paper to develop a more general sensitivity analysis, that goes beyond the non-degenerate case, to improve our understanding of stability to perturbations of low-complexity regularized inverse problems. %In turn this will allow to bridge the gap between theory and practice.

\subsection{Contributions and Outline}

Our first contribution consists in introducing, in Section~\ref{sec-mirror-strat-func},~a class of proper lower-semicontinuous (lsc) convex functions that we coin ``mirror-stratifiable''. The subdifferential of such a function induces a primal-dual pairing of stratifications, which is pivotal to track identifiable strata. We discuss several examples of functions enjoying such a structure. 
%In particular, they are pervasive in imaging sciences and machine learning.
%
With this structure at hand, we then turn to the main result of this paper, formalized in Theorem~\ref{thm:main}, and on which all the others rely. \hl{This result shows finite enlarged activity identification for a mirror-stratifiable function without the need of any non-degeneracy condition. In addition, the identifiable strata are precisely characterized in terms primal-dual optimal solutions.}
Sections~\ref{sec-sentivity-hybrid},~\ref{sec-ip-regul} and~\ref{sec-algorithm}~instantiate this abstract result for a set of concrete problems, respectively: sensitivity of composite \hl{"smooth+non-smooth"} optimization problems (Theorem~\ref{thm-stability-ps}), \hl{enlarged activity} identification by proximal splitting schemes such as Forward-Backward and Douglas-Rachford algorithms (Theorems~\ref{thm-FB}~and~\ref{thm-dr}), and finally, \hl{enlarged identification} for regularized inverse problems (Theorem~\ref{thm-ip}). Before stating these results, we make a short review of the associated literature.
Finally Section~\ref{sec-numerics} illustrates these theoretical findings with numerical experiments, in particular in a compressed sensing scenario involving the $\ell_1$ and nuclear norms as regularizers. 
Following the philosophy of reproducible research, all the code to reproduce the figures of this article is available
online\footnote{Code available at \texttt{https://github.com/gpeyre/2017-SIOPT-stratification}}.

%% file: sections/stratif-functions.tex
\section{Mirror-Stratifiable Functions}
%%%%%%%%%%%%%%%%%%%%%%%%%%%%%%%%%%%%%%%%%%%%%%%%%%%%%%%%%%%%
%%%%%%%%%%%%%%%%%%%%%%%%%%%%%%%%%%%%%%%%%%%%%%%%%%%%%%%%%%%%
\label{sec-mirror-strat-func}

%\todo{G: intro needed}
In this section, we introduce the class of mirror-stratifiable convex functions, and present the sensitivity property they provide (Theorem~\ref{thm:main}). We also illustrate that many popular regularizers used in data science are mirror-stratifiable. Most the results of this section (in particular all the examples) are easy to obtain by basic calculus; we \hl{therefore} do not give these proofs in the text and we gather some of them in Appendix~\ref{appendix}.

%%%%%%%%%%%%%%%%%%%%%%%%%%%%%%%%%%%%%%%%%%%%%%%%%%%%%%%%%%%%
\subsection{Stratifications and definitions}
%%%%%%%%%%%%%%%%%%%%%%%%%%%%%%%%%%%%%%%%%%%%%%%%%%%%%%%%%%%%

%Before introducing the main notion of our work (mirror-stratifiable functions), 
We start with recalling the following standard definition of stratification.

%\todo{G: I really think we need to remove manifold from the definition. This is misleading, and will not help the reader to understand what matters for us.}

\begin{defn}[Stratification]\label{def:strat}
A stratification of a set $D \subset \RR^N$ is a finite partition $\Strat = \{M_i\}_{i\in I}$ such that for any partitioning sets (called strata) $\Man$ and $\Man'$ we have 
\eq{
	\Man \cap \cl(\Man') \neq \emptyset 
	\implies
	\Man \subset \cl(\Man').
}
If the strata are open polyhedra, then $\Strat$ is a polyhedral stratification, and if they are $C^2$-smooth manifolds then $\Strat$ entails a $C^2$-stratification.
\end{defn}

A stratification is naturally endowed with the partial ordering $\leq$ in the sense that
\eql{\label{eq-order-strata}
	\Man \leq \Man' 
	\quad\Longleftrightarrow\quad
	\Man \subset \cl(\Man')
	\quad\Longleftrightarrow\quad
	\Man \cap \cl(\Man') \neq \emptyset .
}
%The first equivalence stems from the definition of ordering and the second one follows from Definition~\ref{def:strat}. 
The relation is clearly reflexive and transitive.
%the skew-symmetry also holds in case of manifold stratification. 
Furthermore, we have
\eql{\label{eq-closure}
	\cl(\Man) = \bigcup_{\Man' \leq \Man} \Man'.
}
%In~words, the closure of a stratum $\Man$ is itself stratifiable with a sub-stratification.% of $\Man$. 
%
%If the strata are polyhedra, we say that $\Strat$ is a polyhedral stratification, if the strata are $C^2$-manifolds, we say that $\Strat$ is a $C^2$ stratification.

%\todo{G: explain that actually the manifold hypothesis is not used in the paper, we we still keep it to be inline with the usual definition, and also because in most case of practical interest, strata are indeed manifolds.}

%\todo{G: the semi-algebraic example is a bit abstract. Maybe give another one, more concrete ? }

An immediate consequence of Definition~\ref{def:strat} is that for each point $x \in D$, there is a unique stratum containing $x$, denoted $\Man_x$. Indeed, suppose that there are two non-empty open strata $\Man_1$ and $\Man_2$ such that $\Man_1 \cap \Man_2 = \ens{x}$, and thus $\Man_1 \cap \Man_2 \neq \emptyset$. This implies, using \eqref{eq-order-strata}, that $\Man_1 \leq \Man_2$ and $\Man_2 \leq \Man_1$, and thus $\Man_1=\Man_2$.\\

At this stage, it is worth emphasizing that the strata are not needed to be manifolds in the rest of the paper. It is however the case that in many practical cases that we will discuss, strata are indeed manifolds, and sometimes affine manifolds.

\begin{exmp}[Polyhedral sets and functions]
A partition of a polyhedral set into its open faces induces a natural finite polyhedral stratification of it. In turn, let $R: \RR^N \to \RRb \eqdef \RR \cup \ens{+\infty}$ be a polyhedral function, and consider a polyhedral stratification of its epigraph, which is a polyhedral set in $\RR^{N+1}$. Projecting all polyhedral strata onto the first $N$-coordinates one obtains a finite polyhedral stratification of $\dom(R)$.
\end{exmp}

\begin{rem}
The previous example extends to semialgebraic sets and functions, which are known to induce stratifications into finite disjoint unions of manifolds. In fact, this holds for any tame class of sets/functions; see, e.g.,~\cite{coste1999omin}.
\end{rem}

%\begin{exmp}[Semi-algebraic functions]
%A semi-algebraic set admits a stratification (with $C^\infty$ algebraic manifolds). More generally, a set with $o$-minimal structure admits a stratification (with smooth manifolds).
%%
%For a semi-algebraic function $f\colon \RR^n\to \RR$, there exists a stratification such that $f$ is smooth on each strata. (warning: when $f$ is furthermore convex, we do not have that $f$ is partly-smooth with respect to the strata)
%%
%Note: for $f$ convex and semi-algebriac, $\partial f$, $\dom f$, $\Im \partial f$ as well as closures and (relative) interiors are all semi-algebraic functions or sets.
%\end{exmp}

%\todo{G: intro needed on Active strata}

We now single out a specific set of strata, called active strata, that will play a central role for \hl{finite enlarged activity identification} purposes.

\begin{defn}[Active strata]\label{eq-active-def}
Given a stratification $\Strat$ of $D \subset \RR^N$ and a point $x \in D$, a stratum $\Man \in \Strat$ is said active at $x$ if $x \in \cl(\Man)$. %The set of active strata at $x$ is denoted \hl{$\Aa_{\Strat}(x)$}.
\end{defn}
\hl{
\begin{prop}\label{prop:strat}
Given a stratification $\Strat$ of $D$ and a point $x \in D$. Then there exists $\delta>0$ such that set of strata $\Man_{x'}$ for any $x'$ such that $\norm{x'-x} < \delta$, coincides with the set of strata $\Man\geq \Man_x$ and with the set of active strata at $x$.
%\begin{align}
%	       \hl{\Aa_{\Strat}(x)}
%	       &= \enscond{ \Man \in \Strat }{ \Man \geq \Man_x }, \label{eq-active-ineq} \\
	       %&= \bigcap_{\delta > 0}\enscond{ \Man_{x'} \in \Strat }{ \norm{x'-x} \leq \delta }. \label{eq-active-neigh}
%	       &= \enscond{ \Man_{x'} \in \Strat }{ \forall x' \text{ such that } \norm{x'-x} < \delta }. \label{eq-active-neigh}
%\end{align}
\end{prop}
}
%%%%%%%%%%%%%%%%%%%%%%%%%%%%%%%%%%%%%%%%%%%%%%%%%%%%%%%%%%%%
\subsection{Mirror-Stratifiable functions}
%%%%%%%%%%%%%%%%%%%%%%%%%%%%%%%%%%%%%%%%%%%%%%%%%%%%%%%%%%%%

Following~\cite[Section~4]{daniilidis2013orthogonal}, we define the key correspondence operator whose role will become apparent shortly.

\begin{defn}
Let $R\colon\RR^N\to\RRb$ be a proper lsc convex function. The associated correspondence operator $\Jj_R: 2^{\RR^N} \to 2^{\RR^N}$ is defined as
\eq{
	\Jj_R(S) \eqdef \bigcup_{x \in S} \ri(\partial R(x)),
}
\end{defn}
where $\partial R$ is the subdifferential of $R$.

%Let a convex set-valued (point-to-set) mapping $A\colon \RR^N \rightrightarrows \RR^N$ with convex images ($A(x)$ is convex for all $x\in \dom A$). 
%We denote by $\Jj_A$ the set-valued "relative interior" operator associated to $A$: the image of a set $S \subset \RR^N$ by $\Jj_A$ is 
%\eq{
%	\Jj_A(S) \eqdef \bigcup_{x \in S} \ri(A(x)) \quad \big(\subset \RR^N\big).
%}
%where $\ri$ is the relative interior of the convex set $A(x)$.

%\begin{lem}\label{lem:partialR}
%If $R$ is stratifiable, then 
%we have ${\Jj_{R}}^{-1} = \Jj_{R^*}$
%\end{lem}
%
%\begin{proof}
%We have $\Im \partial R = \dom \partial R^*$ so the functions have associated domains of definition.
%
%It is easy to prove that 
%\[
%\Man \subset \cl \big(\Jj_{R^*}(\Man^*)\big)
%\]
%
%XXX Mais je n'ai pas mieux... je n'ai pas la preuve du resultat, malheureusement.
%
%XXX attention $u\in \ri\partial R(x) \iff x\in \ri\partial R^*(u)$ est faux en general. 
%
%XXX (note au cas ou: $\cl\dom \partial R = \dom R$).
%
%\end{proof}
Observe that, by definition, $\Jj_R$ is increasing for set inclusion
\eql{\label{eq-J-inclusion}
	S \subset S' \qarrq \Jj_R(S) \subset \Jj_R(S').
}
For this operator to be useful in sensitivity analysis, we \hl{will} further impose that $\Jj_R$ is \textit{decreasing} for the partial ordering $\leq$ in~\eqref{eq-order-strata}, 
%Moreover, $\Jj_R$ must restrict to an invertible mapping between stratifications of $\dom(\partial R)$ and $\dom(\partial R^*)$ with inverse given by $\Jj_{R^*}$. 
as captured in our main definition. \hl{This is a key requirement that captures the intuitive idea that the larger a primal stratum, the smaller its image by $\Jj_R$ in the dual space.} 
%and the reverse ordering of the dual strata maintains a valid stratification of the dual space.}

In the following, we will denote $R^*$ the Legendre-Fenchel conjugate of $R$.

\begin{defn}[Mirror-stratifiable functions]\label{def:mirstrat}
Let $R\colon\RR^N\to\RRb$ be a proper lsc convex function;
we say that $R$ is mirror-stratifiable with respect to  
a (primal) stratification $\Strat=\{\Man_i\}_{i\in I}$ of $\dom(\partial R)$ and a (dual) stratification $\Strat^*=\{\Man^*_i\}_{i\in I}$ of $\dom(\partial R^*)$ if the following holds:
\begin{enumerate}[label=(\roman*)]
\item Conjugation induces a duality pairing between $\Strat$ and $\Strat^*$, and $\Jj_R: \Strat \to \Strat^*$ is invertible with inverse $\Jj_{R^*}$, i.e. $\forall \Man \in \Strat$, $\Man^*\in\Strat^*$ we have
\eq{
\Man^* = \Jj_R(\Man) \iff \Jj_{R^*}(\Man^*) = \Man. 
}
\item $\Jj_R$ is decreasing for the relation $\leq$: for any $\Man$ and $\Man'$ in $\Strat$
\eq{
\Man \leq \Man' \iff \Jj_R(\Man) \geq \Jj_R(\Man'). 
}
\end{enumerate}
\end{defn}

\hl{The primal-dual} stratifications that make $R$ mirror-stratifiable are not unique, as will be exemplified in Remark~\ref{rem:nonunique}. Though the definition is formal and the assumptions looks restrictive, this class include many useful examples, as illustrated in the next section. We finish this section with a remark, used in the sequel.

\begin{rem}[Separability]\label{rem:sep}
\hl{For each $m=1,\ldots,L$, suppose that the proper lsc convex function $R_m: \RR^{N_m} \to \RRb$ is mirror-stratifiable with respect to stratifications $\Strat_m$ and $\Strat_m^*$. Then it is easy to show, using standard subdifferential and conjugacy calculus, that the function $R: (x_m)_{1 \leq m \leq L} \in \RR^{N_1} \times \cdots \times \RR^{N_L} \mapsto \sum_{m=1}^L R_m(x_m)$ is mirror-stratifiable with stratifications $\Strat_1 \times \cdots \times \Strat_L$ and $\Strat_1^* \times \cdots \times \Strat_L^*$.}
\end{rem}

%%%%%%%%%%%%%%%%%%%%%%%%%%%%%%%%%%%%%%%%%%%%%%%%%%%%%%%%%%%%
\subsection{Examples}
%%%%%%%%%%%%%%%%%%%%%%%%%%%%%%%%%%%%%%%%%%%%%%%%%%%%%%%%%%%%

The notion of a mirror-stratifiable function looks quite rigid. However, many of the regularization functions routinely used in data science are mirror-stratifiable. Let us provide some relevant examples in this section. In particular, the $\ell_1$-norm and the nuclear norm will be used in the numerical experiments.

\subsubsection{Legendre functions}
\hl{A lsc convex function $R: \RR^N \to \RRb$ is said to be a Legendre function (see \cite[Chapter~26]{rockafellar1970convex}) if (i) it is differentiable and strictly convex on the interior of its domain $\interop(\dom(R)) \neq \emptyset$, and (ii) $\norm{\nabla R(x_k)} \to +\infty$ for every sequence $\pa{x_k}_{k\in \NN} \subset \interop(\dom(R))$ converging to a boundary point of $\dom(R)$. Many functions in convex optimization are Legendre; most notably, quadratic functions and the log barrier of interior point methods. It was shown in \cite[Theorem~26.5]{rockafellar1970convex} that $R$ is Legendre if and only if its conjugate $R^*$ is Legendre, and that in this case $\nabla R$ is a bijection from $\interop(\dom(R))$ to $\interop(\dom(R^*))$ with $\nabla R^*=(\nabla R)^{-1}$.  As a consequence, a Legendre function $R$ is mirror-stratifiable with $\Strat = \ens{\interop(\dom(R))}$ and $\Strat^* = \ens{\interop(\dom(R^*))}$.}

%%%
\subsubsection{$\ell_1$-norm}
Let $R: x \in \RR \mapsto |x|$, whose conjugate $R^*=\ind_{[-1,1]}$. It follows that $R$ is mirror-stratifiable with $\Strat = \left\{\,]-\infty,0[,\{0\},]0,+\infty[\right\}$ and $\Strat^*= \{\{-1\}, ]-1,1[, \{+1\}\}$. Using Remark~\ref{rem:sep}, the next result is clear.
%Consider the set of signed vectors $S = \ens{-1,0,+1}^N$. We consider the following $3^N$ affine manifolds, defined for $s\in S$ by
%\[
%	\Man_s \eqdef \enscond{x\in \RR^N }{\sign(x)=s }.
%\]
%We observe that these manifolds form a stratification of $\RR^N$.
%We also consider the dual family defined for $s\in S$ by 
%\[
%	\Man^*_s = \enscond{u\in \RR^N }{  u_{\supp(s)}=s_{\supp(s)}, \norm{u_{\supp(s)^c}}_{\infty}<1 }.
%\]
%which forms a stratification of $[-1,+1]^N$. The next result is then clear.

\begin{lem}\label{lem:stratl1}
The $\ell_1$-norm and its conjugate $\ind_{[-1,+1]^N}$ are mirror-stratifiable with respect to the stratifications $\Strat=\left\{\,]-\infty,0[,\{0\},]0,+\infty[\right\}^N$ of~$\RR^N$ and $\Strat^*=\{\{-1\}, ]-1,1[, \{+1\}\}^N$ of $[-1,+1]^N$.
\end{lem}
 
\hl{A graphical illustration of this lemma in dimension $2$ is displayed in Figure~\ref{fig:msl1}(a).} 
%\begin{lem}
%The $\ell_1$-norm and its conjugate the indicator of the unit $\ell_\infty$-ball 
%$\ind_{[-1,+1]^N}$ are mirror-stratifiable with respect to the stratifications $\Strat=\{\Man_s\}_{s\in S}$ of $\RR^N$ and $\Strat^*=\{\Man^*_s\}_{s\in S}$ of $[-1,+1]^N$.
%\end{lem}
%
%\begin{proof}
%The explicit expression of $\partial \norm{\cdot}_1$, ie
%\[
%\partial \norm{\cdot}_1 (x) = 
%	\enscond{u\in \RR^N }{  u_{\supp(s)}=x_{\supp(s)}, \norm{u_{\supp(x)^c}}_{\infty}\leq1 }
%\]
%shows that $\Man^*_s = \Jj_{\norm{\cdot}_1} (\Man_s)$. Conversely, for $u\in \Man^*_s$ , we have that 
%\[
%\partial \ind_{B_\infty} (u) = N_{B_\infty}(u) = \RR_+s,
%\]
%and therefore $\Jj_{R^*}(\Man^*_s) = \RR_+^*s = M_s$. Point (ii) follows from the definition of manifolds.
%\end{proof}

\begin{rem}[Non-uniqueness of stratifications]\label{rem:nonunique}
In general, we do not have uniqueness of stratifications inducing mirror-stratifiable functions. To illustrate this, let us consider the $\ell_1$ norm. Take the partitions $\Strat=\left\{\,]-\infty,0[ \cup ]0,+\infty[,\{0\}\right\}^N$ of~$\RR^N$ and $\Strat^*=\{\{-1,+1\}, ]-1,1[\}^N$. These are valid stratifications though the strata are not connected. Other possible valid stratifications are given by strata $\Man_j = \enscond{x\in \RR^N}{\norm{x}_0=j}$ and $\Man^*_j=\Jj_{\norm{\cdot}_1}(\Man_j)$, for $j=0,1,\ldots,N$.
%the following union of strata: for $j=0,1,\ldots,N$, $\Man_j=\underset{s \in S: \norm{s}_0=j}{\bigcup} \Man_s$ and $\Man^*_j=\underset{s \in S: \norm{s}_0=j}{\bigcup} \Man^*_s$. The strata $M_j$ and $M^*_j$ are not connected, yet they form valid stratifications.
%
It is immediate to check that the $\ell_1$-norm is also mirror-stratifiable with respect to these stratifications. However, as devised above, it is in general not wise to take such large strata as they lead to less sharp localization and sensitivity results.
%
%\todo{G: maybe add that in general fusing strata is not a smart move, since it leads to a less sharp identification analysis. }
\end{rem}

\begin{rem}[Instability under the sum rule]\label{rem:sumrule}
The family of mirror-stratifiable functions is unfortunately not stable under the sum. As a simple counter-example, consider the pair of conjugate functions on $\RR$: $R(x)= |x| + x^2/2$ and $R^*(u)=(\max\{|u|-1,0\})^2/2$. We obviously have $\dom(\partial R) =\dom(\partial R^*)=\RR$. However we observe that $\Jj_{R}(\RR) = \RR\setminus \{-1,1\}$. This yields that $\Jj_R$ cannot be a pairing between any two stratifications of $\RR$, and therefore $R$ cannot be mirror-stratifiable. 
%This is not very surprising, since the convex conjugacy is a duality between strict convexity and smoothness.
\end{rem}

\begin{figure}
\centering
\begin{tabular}{@{}c@{}}
\includegraphics[width=0.67\linewidth]{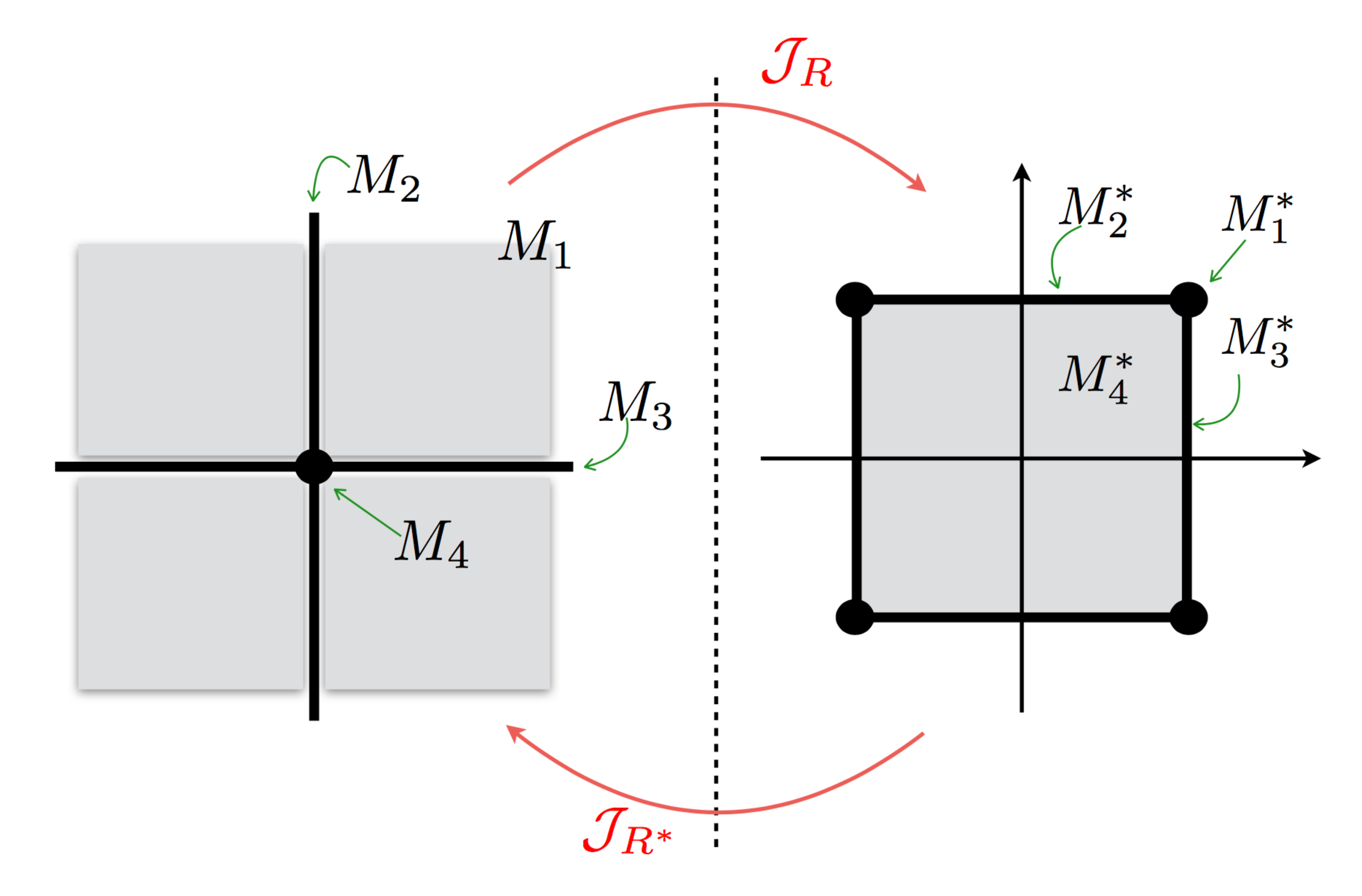} \\
(a) \\
\includegraphics[width=0.8\linewidth]{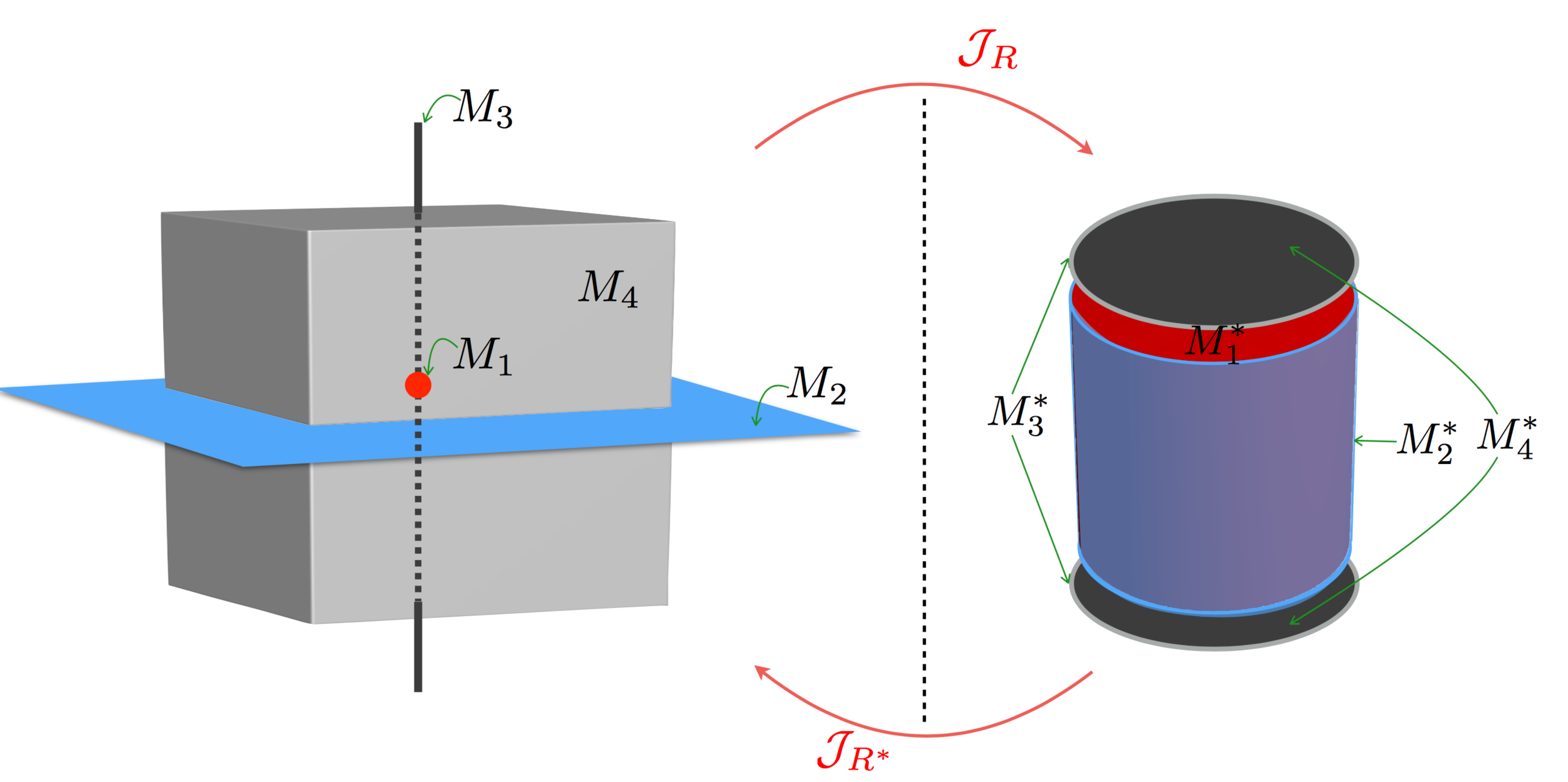}\hspace*{1.4cm} \\
(b)
\end{tabular}
\caption{
\hl{Graphical illustration of mirror-stratification for two norms: (a) $\ell_1$ norm in dimension $2$; (b) $\ell_{1,2}$ norm in dimension $3$ (with two blocks, of respectively sizes $1$ and $2$).}
}
\label{fig:msl1}
\end{figure}

%%%
\subsubsection{$\ell_{1,2}$ norm}\label{sec:l12}

%\todo{G: Write-me!}
The $\ell_{1,2}$ norm, also known as the group Lasso regularization, has been advocated to promote group/block sparsity~\cite{yuan2005model}, i.e. it drives all the coefficients in one group to zero together. Let $\Bb$ a non-overlapping uniform partition of $\{1,\ldots,N\}$ into $K$ blocks. The $\ell_{1,2}$ norm induced by the partition $\Bb$ reads
\[
\norm{x}_{\Bb} = \sum_{B\in \Bb} \norm{x^B}_2
\]
where $x^B$ is the restriction of $x$ to the entries indexed by the block $B$. This is again a separable function of the $x^B$'s. One can easily show that the $\ell_2$ norm on $\RR^{|B|}$ and its conjugate, the indicator of the unit $\ell_2$-ball $\Ball_{\ell_2^{|B|}}$ in $\RR^{|B|}$, are mirror-stratifiable with respect to the stratifications $\ens{\ens{0},\RR^{|B|} \setminus \ens{0}}$ of $\RR^N$ and $\ens{\SS_{\ell_2}^{|B|-1},\interop(\Ball_{\ell_2^{|B|}})}$ of $\Ball_{\ell_2^{|B|}})$, where $\SS_{\ell_2}^{|B|-1}$ is the corresponding unit sphere. In turn, the $\ell_{1,2}$ norm is mirror-stratifiable with respect to the stratifications $\Strat=\ens{\ens{0},\RR^{|B|} \setminus \ens{0}}^K$ and $\Strat^*=\ens{\SS^{|B|-1},\interop(\Ball_{\ell_2^{|B|}})}^K$. \hl{An illustration of this result in dimension $3$ is portrayed in Figure~\ref{fig:msl1}(b).}

%Let $\Bb$ a partition of $\{1,\ldots,N\}$, we define the associated norm by
%\[
%\norm{x}_{\Bb} = \sum_{B\in \Bb} \norm{x^B}_2
%\]
%
%Define the support by block as $\supp x =\enscond{ B\in \Bb }{ x^B \neq 0 }$, and we have :
%\[
%	\Man = \enscond{ x }{ \supp x = s }
%\]
%\[
%	\Man^* =  \enscond{ u }{
%		 \norm{u^B} = 1 \text{ if }B\not\in s\text{ and }\norm{u^B} < 1 \text{ if }B\in s
%		}
%\]

%%%
\subsubsection{Nuclear norm}\label{subsubsec:stratnuc}

%\todo{Ref to intro where nuclear norm is introduced.}
Let us come back to the nuclear norm defined in Example~\ref{ex:nuc}. For simplicity, we assume $n_1=n_2=n$. Let $\Man$ be a stratum of the $\ell_1$ norm stratification. In the following, we denote by $\Man^{\sym}$ its symmetrization. We observe that $\si^{-1}(\Man^{\sym}) = \enscond{X \in \RR^{n \times n}}{ \rank(X) = \norm{z}_0, z \in \Man}$. So many inverse images $\si^{-1}(\Man^{\sym})$ of strata $\Man^{\sym}$ coincide. This suggests the following $n+1$ stratifications: for a given $i\in \{0,\ldots,N\}$
\begin{eqnarray*}
\Man_i &=& \enscond{X \in \RR^{n \times n}}{ \rank(X) = i }\\
\Man^*_i &=& \enscond{U \in \RR^n}{\si_1(U) = \cdots = \si_i(U) = 1, \forall j>i,\ |\si_j(U)|<1 }.
\end{eqnarray*}
This yields that $\norm{\cdot}_*$ is mirror-stratifiable with respect to these stratifications.

%\begin{rem}
%The nuclear norm is actually partly smooth relative to each stratum $\Man_i$, which turns out to be a $C^2$-smooth manifold; see~\cite{daniilidis2013orthogonal}.
%%(i.e. $R$ is partly smooth at each $x \in \RR^N$ with respect to the unique manifold $\Man_x \in \Strat$ such that $x \in \Man_x$).
%%For a mirror-stratifiable function that do not belong to the two previous classes, see later ( $\ell_{1,2}$).
%\end{rem}
%

%%%%%%%%%%%%%%%%%%%%%%%%%%%%%%%%%%%%%%%%%%%%%%%%%%%%%%%%%%%%%%%%%
\subsubsection{Polyhedral Functions}
We here establish mirror-stratifiability of polyhedral functions,
% which encompasses many regularizers widely used in image sciences, computer vision, machine learning and statistics,
including the $\ell_1$ norm, the $\ell_{\infty}$ norm and anisotropic TV semi-norm. 

A polyhedral function $R \colon \RR^N\to\RRb$ can be expressed as
\begin{equation}\label{eq:poly}
R(x) = \max_{i=1,\ldots,k} \{\dotp{a_i}{x} -\al_i \}
+ \ind_{\bigcap_{i=k+1,\ldots,m} \enscond{x }{ \dotp{a_i}{x} -\al_i \leq 0} }(x).
\end{equation}
%We recall the classical fact that the conjugate of a polyhedral function is also polyhedral. Moreover
%\[
%	\dom(R) = \dom(\partial R) = \bigcap_{i=k+1,\ldots,m} \enscond{x }{ \dotp{a_i}{x} -\al_i \leq 0}.
%\]
%%and standard subdifferential calculus results give that $\dom \partial R = \dom R$.
For any $x\in \dom(R)$, introduce the two sets of indices
\begin{eqnarray*}
\Imax(x)&=& \enscond{i=1,\ldots,m}{ \dotp{a_i}{x}-\al_i =R(x) }, \\
\Ifeas(x)&=& \enscond{i=1,\ldots,m}{ \dotp{a_i}{x}=\al_i }.
\end{eqnarray*}
For a given index set $I\subset\{1,\ldots,m\}$, we consider the affine manifold
\eql{\label{eq-MI-poly}
\Man_I \eqdef\enscond{x\in \dom(R)}{ \Imax(x)\cap\Ifeas(x) = I }.
}
We see that some $\Man_I$ may be empty and that $\Strat = \ens{\Man_I}_I$ is a stratification of $\dom(R)$.
%\footnote{In fact, we note that $R$ is partly-smooth with respect to it - but we do not needed this I guess.
%We also get from the definition \eqref{eq-MI-poly} of $\Man_I$ that its closure has the simple expression
%\[
%	\cl(\Man)_I = \enscond{x\in \RR^N }{ \Imax(x)\cap\Ifeas(x) \supset I };
%\]
%this corresponds to \eqref{eq-closure}. Not sure we need this either.}
The stratum $\Man_I$ is characterized by the optimality part $\Imax=I\cap \{1,\ldots,k\}$ and the feasibility part $\Ifeas=I\cap \{k+1,\ldots,m\}$ of $I$. Similarly we define
\[
\Man^*_I
= \ri(\conv\enscond{a_i}{i\in \Imax}) + \ri(\cone\enscond{a_i}{i\in \Ifeas}).
\]
Let us formalize in the next proposition a result alluded to in \cite{daniilidis2013orthogonal}.

\begin{prop}\label{prop-polyhedral}
A polyhedral function $R$ is mirror-stratifiable with respect to its naturally induced stratifications $\{\Man_I\}_{I}$ and $\{\Man^*_I\}_{I}$.
\end{prop}

%%%
\begin{rem}[Back to $\norm{\cdot}_1$]
As we anticipated, Proposition~\ref{prop-polyhedral} subsumes Lemma~\ref{lem:stratl1} as a special case. To see this, observe that
\[
\norm{x}_1=\max_{\norm{u}_{\infty}\leq 1} \dotp{u}{x} = \max_{u\in \{-1,+1\}^N} \dotp{u}{x} ,
\]
which is of the form~\eqref{eq:poly} with $k=m=2^N$. Thus there are $2^{2^N}$ affine manifolds as defined by \eqref{eq-MI-poly}. But many of them are empty: there is only $3^N$ (non-empty, distinct) manifolds in the stratification of $\RR^N$, and they coincide with those in Lemma~\ref{lem:stratl1}.
%There are $2^{N-\supp(\bar x)}$ $y$ attaining $\norm{\bar x}_{1}$.
\end{rem}

%%%%%%%%%%%%%%%%%%%%%%%%%%%%%%%%%%%%%%%%%%%%%%%%%%%%%%%%%%%%%%%%%
\subsubsection{Spectral Lifting of Polyhedral Functions}
%%%%%%%%%%%%%%%%%%%%%%%%%%%%%%%%%%%%%%%%%%%%%%%%%%%%%%%%%%%%%%%%%

%\todo{Motivate a bit why spectral lifting is relevant/important in practice.}
As we discussed in Remark~\ref{rem:sumrule}, $\Jj_R$ may fail to induce a duality pairing between stratifications of $R$ and $R^*$, in which case $R$ cannot be mirror-stratifiable. Hence, for this type of duality to hold, one needs to impose stringent strict convexity conditions. To avoid this, and still afford a large class of mirror-stratifiable functions that are of utmost in applications, we consider spectral lifting of polyhedral functions, in the same vein as \cite{daniilidis2013orthogonal} did it for partial smoothness.

A matrix function $R\colon\RR^{N=n\times n}\to\RRb$ is said to be a spectral lift of a polyhedral function if there exists a polyhedral function $R^{\sym}\colon\RR^n\to\RR$, invariant under signed permutation of its coordinates, such that $R =  R^{\sym} \circ \si$ where $\si$ computes the singular values of a matrix.
Associated to $\Strat=\{\Man_I\}_I$ the (polyhedral) stratification induced by $R^{\sym}$, we consider its symmetrized stratification. We define 
the symmetrization of $\Man\in \Strat$, as the set $\Man^{\sym}$
\[
\Man^{\sym} =
\enscond{x\in \RR^N}{\exists y \in \Man \text{ such\;that\;for\;all\;$i$\;there\;exists\;$j$\;with\;} |x^i|= |y^j|} .
\]
The next result is a corollary of the main result of \cite{daniilidis2013orthogonal}.

\begin{prop}\label{prop:stratspectral}
A spectral function $R = R^{\sym}\circ \si$ is mirror-stratifiable with respect to the smooth stratification $\{\si^{-1}(\Man^{\sym})\}$ and its image by $\Jj_R$.
\end{prop}

%%%
\begin{rem}[Back to $\norm{\cdot}_*$]
The nuclear norm is a spectral lift of the $\ell_1$ norm. Therefore, one can recover mirror-stratification of the nuclear norm, with the stratifications given in Section~\ref{subsubsec:stratnuc}, by putting together Lemma~\ref{lem:stratl1} and Proposition~\ref{prop:stratspectral}.
\end{rem}

\subsection{\hl{Activity} Identification for Mirror-Stratifiable Functions}
%%%%%%%%%%%%%%%%%%%%%%%%%%%%%%%%%%%%%%%%%%%%%%%%%%%%%%%%%%%%

To our point of view, the notion of mirror-stratifibility deserves a special study in view of the following simple but powerful geometrical observation. We state it as a theorem because of its utmost importance in the subsequent developments of this paper. 

\begin{thm}[\hl{Enlarged activity} identification]\label{thm:main}
Let $R$ be a proper lsc convex function which is mirror-stratifiable with respect to primal-dual stratifications $\Strat=\{\Man\}$ and $\Strat^*=\{\Man^*\}$. Consider a pair of points $(\xlim,\ulim)$ and the associated strata 
$\Man_{\xlim}$ and $\Man^*_{\ulim}$.
If the sequence pair $(x_k,u_k) \to (\xlim,\ulim)$ is such that $u_k \in \partial R(x_k)$, 
then for $k$ large enough, $x_k$ is localized in a specific set of strata such that 
\begin{equation}\label{eq-ident}
\Man_{\xlim} \leq \Man_{x_k} \leq \Jj_{R^*}(\Man_{\ulim}^*) .
\end{equation}
\vspace*{-3ex}
\end{thm}
\begin{proof}
By assumption on $R$, $\partial R$ is sequentially closed~\hl{\cite[Theorem~24.4]{rockafellar1970convex}} and thus $\ulim\in \partial R (\xlim)$. 
Now, since $x_k$ is close to $\xlim$, upon invoking~Proposition\;\ref{prop:strat}, we get
\eq{
	\Man_{x_k} \geq \Man_{\xlim}
}
which shows the left-hand side of~\eqref{eq-ident}. Similarly, we have
\eql{\label{eq-continuity-inclusion-dual}
	\Man^*_{u_k} \geq \Man^*_{\ulim}.
}
Using the fact that $\partial R(x_k)$ is a closed convex set together with~\eqref{eq-J-inclusion} leads to
\eq{	
	u_k \in \partial R(x_k) = \cl(\ri( \partial R(x_k) ))
	= \cl( \Jj_R(\{x_k\}) )
	\subset \cl( \Jj_R(\Man_{x_k}) )
}
which entails $\Man_{u_k} \leq \Jj_R(\Man_{x_k})$. Using finally~\eqref{eq-continuity-inclusion-dual} and that by definition, $\Jj_R$ is decreasing for the relation $\leq$, we get
\eq{
	\Man^*_{\ulim} \leq \Man^*_{u_k} \leq \Jj_R(\Man_{x_k})
	\qarrq
	\Man_{x_k} \leq \Jj_{R^*}(\Man^*_{\ulim}) , 
}
whence we deduce the right-hand side of~\eqref{eq-ident}.
\end{proof}

Mirror-stratification thus allows us to prove the simple but powerful claim of Theorem~\ref{thm:main}. As we will see in the rest of the paper, this result will be the backbone to prove sensitivity and finite \hl{enlarged activity} identification results in absence of non-degeneracy. %type conditions, which, to the best of our knowledge, are new to the literature.

%more precisely, we do not cover the general identification and sensibility results of the literature, but we cover all the illustrative instantiation of these results.

%Existing identification and sensitivity results ensures that the optimal solutions stay on an active manifold, a structure is assumed on the objective function, see  see lewis + vous. Roughly speaking, the function is assumed "V-shaped" normally to the manifold (more precisely partly-smooth) and the minimum is assumed strong. The assumption on the function is commonly satisfied by all the useful functions in inverse problems and machine learning applications. However the non-degeneracy of the solution often fails in practical situations. The developments of this paper are aimed to cover this degenerate case. To have a control outside of the manifold we need to have a global structure: the minor-stratifications provide a simple and powerful setting, as the core of the results is just the previous result, and the range of applications covers all the important examples convex problems in learning machine and image and signal processing.

%% file: sections/composite-sensitivity.tex
% !TEX root = ../StratificationSensitivity.tex

%%%%%%%%%%%%%%%%%%%%%%%%%%%%%%%%%%%%%%%%%%%%%%%%%%%%%%%%%%%%
\section{Sensitivity of Composite Optimization Problems}
%%%%%%%%%%%%%%%%%%%%%%%%%%%%%%%%%%%%%%%%%%%%%%%%%%%%%%%%%%%%
\label{sec-sentivity-hybrid}

In this section, we consider a parametric convex optimization problem of the form
\eql{\label{eq-primal}\tag{$\Pp(p)$}
	\umin{x \in \RR^N} E(x,p) \eqdef F(x,p) + R(x) ,
}
depending on the parameter vector $p \in \Param$, where $\Param$ is the parameter set, an open subset of a finite dimensional linear space. 
Sensitivity analysis studies the properties of solutions $\xsol(p)$ of \eqref{eq-primal} (assuming they exist) to perturbations of the parameters vector $p \in \Param$ around some reference point $\porig$.
In Section~\ref{sec-pw-sensitivity}, we briefly review the existing results on this topic and their limitations. We then introduce in Section~\ref{sec-pw-sensitivity-2} our new results obtained owing to the mirror-stratifiable structure.

%%%%%%%%%%%%%%%%%%%%%%%%%%%%%%%%%%%%%%%%%%%%%%%%%%%%%%%%%%%%%%%%%%%%%%%%%%%%%%%
\subsection{Existing sensitivity results}
\label{sec-pw-sensitivity}

% The regularized inverse problem of the form~\eqref{eq-variational-ip-intro-noisy} is a typical example where the parameters vector is $p=(\la,y) \in \Param=\RR_+ \times \RR^P$ and $\porig=(0,\yorig)$.
%
Classical sensitivity results (see e.g.~\cite{bonnans2000perturbation,mordukhovich1992sensitivity,dontchev1983pertub}) study the regularity of the set-valued map $p \mapsto \xsol(p)$. A typical result proves Lipschitz continuity of this map provided that $E$ is smooth enough and $E(\cdot,\porig)$ has a local second-order (quadratic) growth at $\xsol(\porig)$, i.e. that there exists some $c > 0$ such that $E(x,\porig) \geq E(\xsol(\porig),\porig) + c \norm{x-\xsol(\porig)}^2$ for $x$ nearby $\xsol(\porig)$.
For $C^2$-smooth optimization, this growth condition is equivalent to positive definiteness of the hessian of $E$ with respect to $x$ evaluated at $(\xsol(\porig),\porig)$~\cite{Fiacco68}.
% but necessitates non-trivial second-order generalized differentiation when it comes to non-smooth optimization~\cite{mordukhovich2006var,dontchev1994implicit}. 
For classical smooth constrained optimization problems, activity is captured by the subset of active inequality constraints. Under reasonable non-degeneracy conditions (see, for example, \cite{Fiacco68}), this active set is stable under small perturbations to the objective.

A nice nonsmooth sensitivity theory is based on the notion of \textit{partial smoothness}~\cite{lewis2002active}. Partial smoothness is an intrinsically geometrical assumption which, informally speaking, says that $E$ behaves smoothly along an active manifold and sharply in directions normal to the manifold. 
%
%Partial smoothness is closely related to the notion of identifiable surfaces~\cite{Wright-IdentSurf} and $\Uu\Vv$-decompositions~\cite{Lemarechal-ULagrangian}. 
Furthermore, under a non-degeneracy assumption at a minimizer (see~\eqref{eq-ri-condition}), it allows appealing statements of second-order optimality conditions (including second-order generalized differentiation) and associated sensitivity analysis around that minimizer~\cite{lewis2002active,lewis2013partial}.
%It also shares connections with notions from sensitivity analysis such as tilt-stability~\cite{mordukhovich2012second,lewis2013partial,drusvyatskiy2013tilt}.

%%
%Partly smooth functions form a rich and practical class of non-smooth functions. Indeed, partial smoothness is enjoyed by most functions involved in structured optimization problems one encounters in data sciences, see the examples of regularizers $R$ discussed in Section~\ref{sec-motivnum}. It also enjoys a powerful calculus.
%%
%In particular, partial smoothness lifts nicely through the spectral mapping, i.e. a function of a matrix defined as a permutation-invariant partly smooth function of the singular values of the matrix inherits partial smoothness. Smoothness properties of spectral manifolds~\cite{daniilidis2008prox,DaniilidisMalick14} and spectral lifting of partial smoothness~\cite{daniilidis2013orthogonal} are now well-understood.
%
%Let $\partial E$ be the subdifferential of $E$ according to $x$ and $\ri$ denotes the relative interior of a convex set. 
\hl{Let $\partial E(x,p)$ be the subdifferential of $E$ according to $x$}. Specializing the result of \cite[Proposition~8.4]{drusvyatskiy2013optimality} to a proper lsc convex function $E(\cdot,p)$, one can show that $C^2$-partial smoothness of $E(\cdot,p)$ at $\xsol(\porig)$ relative to some fixed manifold (independent of $p$) for $0$, together with the non-degeneracy assumption 
\eql{\label{eq-ri-condition}
	0 \in \ri(\partial E(\xsol(\porig),\porig))
}
is equivalent to the existence of an identifiable $C^2$-smooth manifold, i.e. for $\xsol(p)$ and $\usol(p) \in \partial E(\xsol(p),p)$ close enough to $\xsol(\porig)$ and $0$, $\xsol(p)$ lives on the active/partly smooth manifold of $\xsol(\porig)$. If these assumptions are supplemented with a quadratic growth condition of $E$ (see above) along the active manifold, then one also has $C^1$ smoothness of the single-valued mapping $p \mapsto \xsol(p)$~\cite[Theorem~5.7]{lewis2002active}. It can be deduced from \cite[Corollary~4.3]{drusvyatskiygeneric16} that for almost all linear perturbations of lsc convex semialgebraic functions, the non-degeneracy and quadratic growth conditions hold. 
%\todo{J : Jalal, je te laisse simplifier ce passage. Il faut dire les limites des resultats existants, mais sans referer au probleme inverses. Pour ne pas tout melanger, je propose de separer parametrique $p$ et parametrique $(\la,y)$. Un peu comme dans mon talk.}
\hl{However, this genericity fails to hold} for many cases of interest. As an example, consider $E$ of~\eqref{eq-variational-ip-intro-noisy} with $\la > 0$ fixed and $p=y$.
%which can also be written as
%\[
%\umin{x \in \RR^N}  \la R(x) + \frac{1}{2}\norm{\Phi x}^2 - \dotp{\Phi^*y}{x} .
%\]
If $R$ is a proper lsc convex and semialgebraic function, one has from \cite{drusvyatskiygeneric16} that for Lebesgue almost all $\Phi^*y$, problem \eqref{eq-variational-ip-intro-noisy} has at most one minimizer at which furthermore non-degeneracy and quadratic growth hold. Of course genericity in terms of $\Phi^*y$ does not imply that in terms of $y$, which is our parameter of interest. Not to mention that we supposed $\la$ fixed while it is not in many cases of interest. 
%In conclusion, \eqref{eq-ri-condition} might fail to hold for regularized ill-posed inverse problems. 

%\todo{G: isn't there a relationship between our result and the ``sticky face lemma'' for the polyhedral case? }
%
%\todo{JF: there is no really interesting relationship. In fact, the sticky face lemma (SFL) is a localization/stability result on exposed faces of polyhedral sets, while our result is a localization/stability result of active sets. Here is a simple example. 
%Take $\norm{\cdot}_1=\sigma_C$, where $C=\ell_\infty$-ball. Consider the problem at $\xlim=0$ and take a sequence $x_k \to 0$. Then as $C=\partial \sigma_C(0)$, the SFL tells us that for $k$ large enough, $\partial \sigma_C(x_k) = \partial \sigma_{\partial \sigma_C(0)}(x_k) = \partial \sigma_C(x_k)$, yielding a trivial conclusion. Our result is Theorem~\ref{thm:main} is much more informative and intuitive. Indeed, take in $\RR^2$ $x_k=(0,1/k) \to \xlim$, and $\ulim = (0,1) \leftarrow u_k = (1/k,1) \in \partial \sigma_C(x_k) = [-1,1] \times \ens{1} \subset \rbd(\partial \sigma_C(\xlim))$ (i.e. we are in the degenerate case). Thus $\Man_{\xlim}=\ens{0}$ and $\Man^*_{\ulim}=]-1,1[ \times \ens{1}$. In turn $\Jj_{\ind_C}(\Man^*_{\ulim})=\ens{0} \times \RR_{+,*}$. Therefore, as expected, we indeed have $\ens{0} \leq \Man_{x_k} \leq \ens{0} \times \RR_{+,*}$ for $k$ large enough.}

%%%%%%%%%%%%%%%%%%%%%%%%%%%%%%%%%%%%%%%%%%%%%%%%%%%%%%%%%%%%%%%%%%%%%%%%%%%%%%%
\subsection{Sensitivity analysis without non-degeneracy}
\label{sec-pw-sensitivity-2}

%We consider a parametric convex optimization problem of the form
%\eql{\label{eq-primal}\tag{$\Pp(p)$}
%	\umin{x \in \RR^N} E(x,p) \eqdef F(x,p) + R(x) ,
%}
%depending on the parameter vector $p \in \Param$, where $\Param$ is the parameter set, an open subset of a finite dimensional linear space.
%\footnote{Applications where $p$ lies in some subset of $\Param$ can be handled by setting $E(x,p)=+\infty$ for $p$ outside that subset.}.
%an open sub-set of a finite dimensional linear space. 
For a fixed $p$, \eqref{eq-primal} is a standard composite optimization problem. Here, we assume that the objective is the sum of a $C^1(\RR^N)$ convex function $F(\cdot,p)$ and a nonsmooth proper lsc convex function $R$. %We assume that this problem is well-posed, i.e. it admits at least one minimum. 
We denote $\nabla F(x,p)$ the gradient of $F(\cdot,p)$ at $x$.

We are going to show that if the minimizer $\xsol(\porig)$ is unique, slight perturbations $p$ of $\porig$ generate solutions $\xsol(p)$ that are in a ``controlled'' stratum $\Man_{\xsol(p)}$ precisely sandwiched to extreme strata defined from a primal-dual pair associated to $\porig$.
%Its dual problem reads
%\eql{\label{eq-dual}
%	\umin{u \in \RR^N} f^*(-u,p) + R^*(u).
%}

%Not sure we need the dual optimality condition 
%$\xsol(p) \in -\partial f^*(-\usol(p),p) ~\text{and}~ \xsol(p) \in \partial R^*(\usol(p))$.	
%
%We also assume that $E$ is bounded below. For simplicity, we can assume that it is nonnegative ?
%In fact what is really necessary is that for all $p$ the infimum of of $E(\cdot,p)$ is finite. 

\begin{thm}[Sensitivity analysis with mirror-stratifiable functions]\label{thm-stability-ps}
	Let $\porig$ be a given point in the parameter space $\Param$. Assume that: (i) $E(\cdot,\porig)$ has a unique minimizer $\xsol(\porig)$, (ii) $E$ is lsc on $\RR^N \times \Param$, (iii)  
	%\todo{on what domain?}, \todo{I think that this hypothesis is needed} that 
	$E(\xsol(\porig),\cdot)$ is continuous at $\porig$, (iv) $\nabla F$ is continuous at $(\xsol(\porig),\porig)$, and (v) $E$ is level-bounded\footnote{
Recall from \cite[Definition~1.16]{rockafellar1998var} that the function $E: \RR^N \times \Param$ is said to be level-bounded in $x$ locally uniformly in $p$ around $\porig$ if for each $c \in \RR$, there exists a neighbourhood $\Vv$ of $\porig$ and a bounded set $\Omega$ such that the sublevel set
$
\enscond{x \in \RR^N}{E(x,p) \leq c} \subset \Omega 
$
for all $p \in \Vv$.}
 in $x$ uniformly in $p$ locally around $\porig$.
	If $R$ is mirror-stratifiable according to $(\Strat$,$\Strat^*)$,
	%$c \in E^\star$ 
	%\eql{\label{eq-uniform-coercivity}
	%	\enscond{x}{ E(p,x) \leq E^\star }
%		\exists C, \foralls (p,x), \quad
%		E(p,x) \geq C \norm{x}.
	%}
	%is bounded independently of $p \in \Param$ \todo{in a neighborhood of $\porig$}.
	%
	%We also assume that $R$ is mirror-stratifiable according to $(\Strat$,$\Strat^*)$.
	%
	then for all $p$ close to $\porig$, any minimizer $\xsol(p)$ of $E(\cdot,p)$ is localized as follows
	\eql{\label{eq-ident-manif-hybrid}
		\Man_{\xsol(\porig)} \leq \Man_{\xsol(p)} \leq \Jj_{R^*}(\Man^*_{\usol(\porig)})
		\qwhereq
		\usol(\porig) \eqdef -\nabla F(\xsol(\porig),\porig).
	}
\vspace*{-2ex}
\end{thm}
\begin{proof}
Le $\pa{p_k}_{k \in \NN} \subset \Param$ be a sequence of parameters converging to $\porig$. %Denote $\xsol_k$ any minimizer of $E(\cdot,p_k)$.
By assumption on $F$ and $R$, $E(\cdot,p)$ is proper for any $p$. Since it is also lsc and is level-bounded in $x$ uniformly in $p$ locally around $\porig$ by conditions (ii) and (v). It follows from \cite[Theorem~1.17(a)]{rockafellar1998var} that $\Argmin E(\cdot, p_k)$ is non-empty and compact, and, in turn, any sequence $\pa{\xsol_k}_{k \in \NN}$ of minimizers is bounded. We consider \hl{a subsequence, which for simplicity, we denote again} $\xsol_k \rightarrow \xlim$. 
We then have
\begin{align*}
E(\xlim,\porig) &\underset{\text{condition (ii)}}{\leq} \liminf_k E(\xsol_k,p_k) \\
			 &\underset{\text{Optimality}}{\leq} \liminf_k E(\xsol(\porig),p_k) \\
			 &\underset{\text{condition (iii)}}{=} \lim_k E(\xsol(\porig),p_k) = E(\xsol(\porig),\porig) .
\end{align*}
%By optimality of $\xsol_k$, one has $E(\xsol_k,p_k) \leq E(\xsol(\porig),p_k)$. One can pass to the limit in this inequality, because (i) for the left-hand side, $E$ ; (ii) for the right-hand side, $E(\xorig,\cdot)$ is continuous at $\porig$. 
%This leads to $E(x^\star,\porig) \leq E(\xorig,\porig)$. 
By the uniqueness condition (i), we conclude that $\xlim=\xsol(\porig)$. 
Let $\usol_k \eqdef -\nabla F(\xsol_k,p_k)$. 
Since $\nabla F$ is continuous at $(\xorig,\porig)$ by assumption (iv), one has $\usol_k \rightarrow \uorig$.
The first-order optimality condition of problem~\eqref{eq-primal} reads $\usol_k \in \partial R(\xsol_k)$. Hence we are now in position to invoke Theorem~\ref{thm:main} to conclude.
\end{proof}

\begin{figure}
\centering
\begin{tabular}{@{}c@{}}
\includegraphics[width=0.7\linewidth]{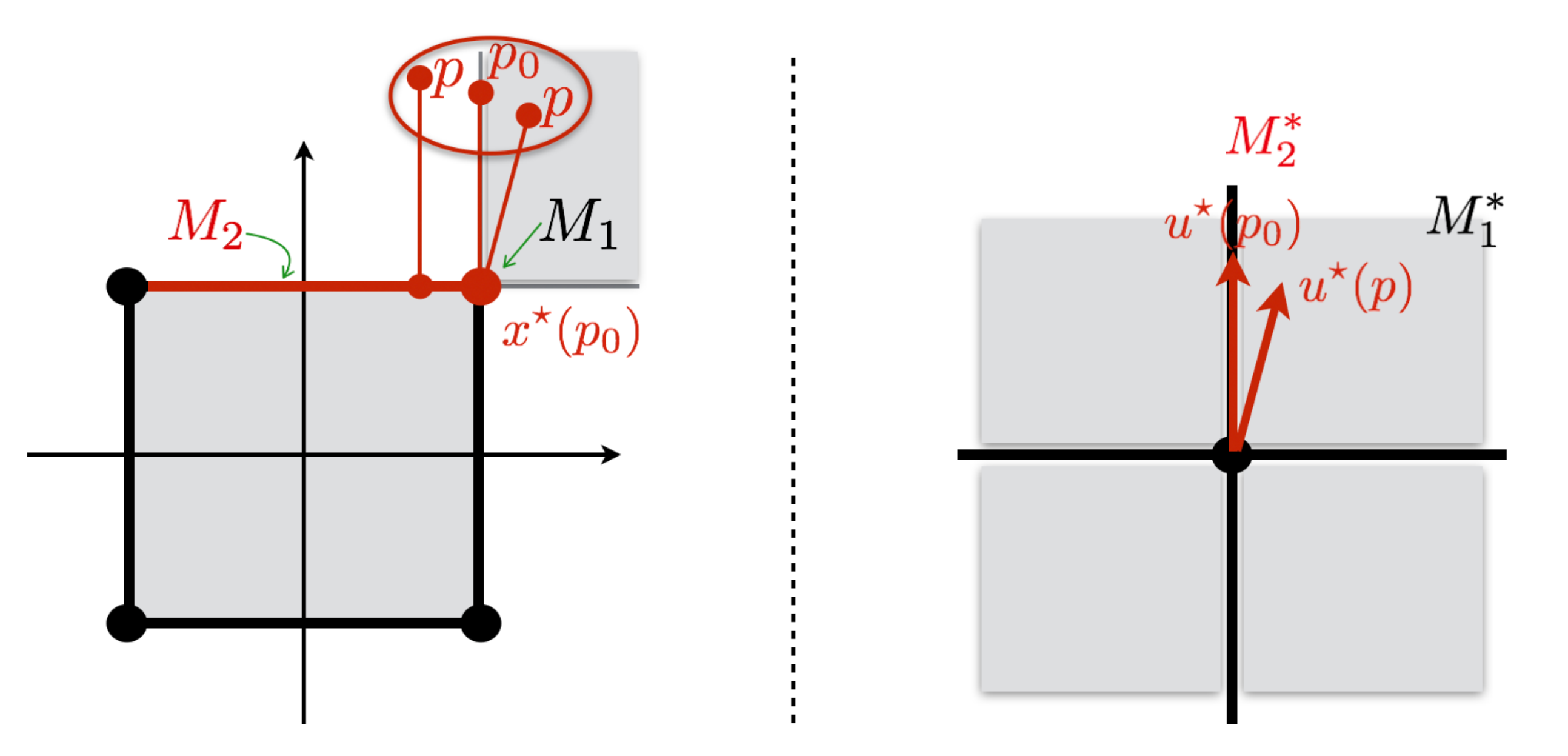} \\
(a) \\[-0ex]
\includegraphics[width=0.7\linewidth]{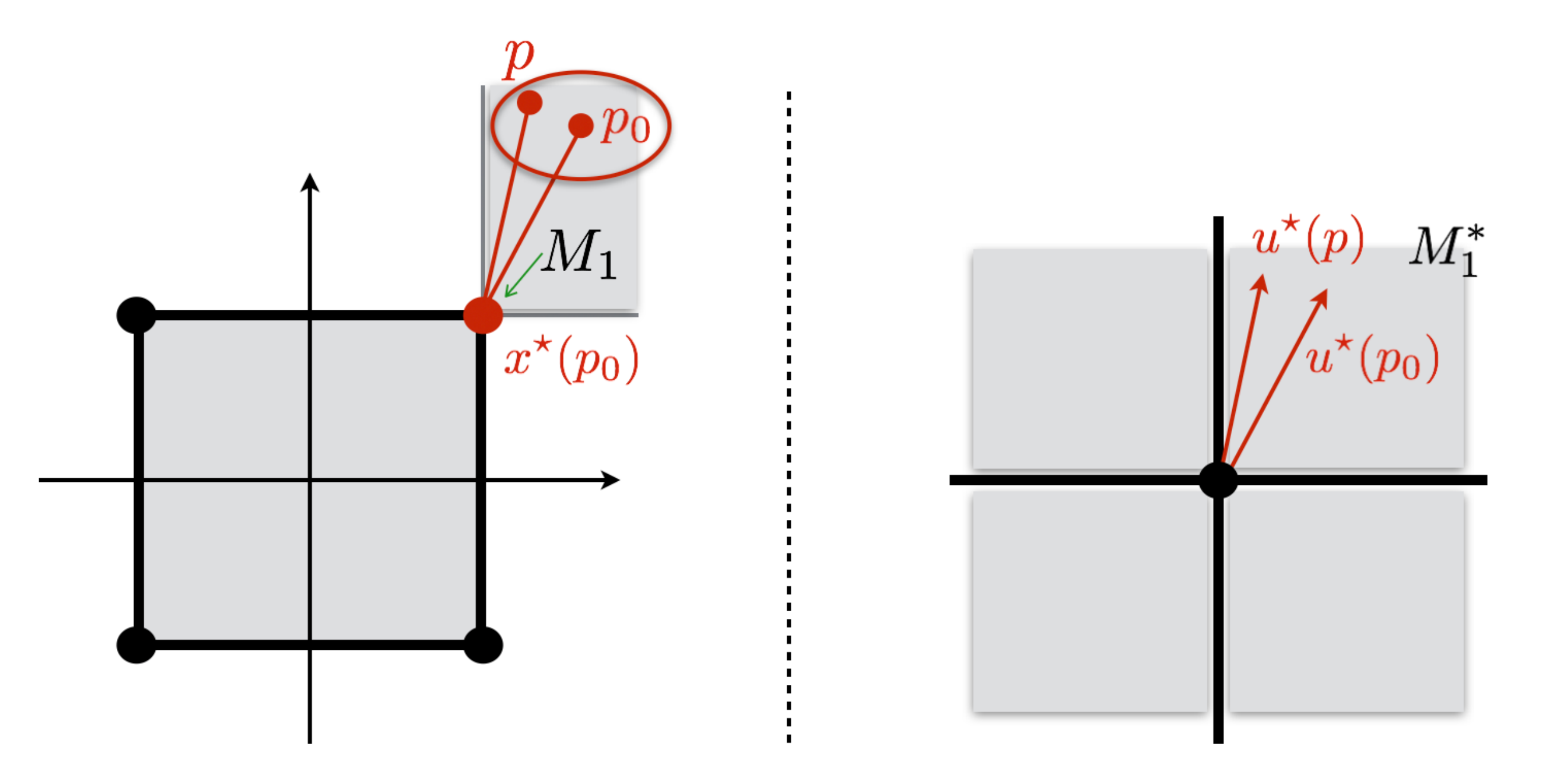} \\
(b)
\end{tabular}
\caption{
\hl{Graphical illustration of Theorem~\ref{thm-stability-ps} for a projection problem on the $\ell_\infty$ ball. (a) corresponds to a degenerate situation while (b) to a non-degenerate one.}
}
\label{fig:projsens}
\end{figure}

\begin{rem}[Single manifold identification]\label{rem:compnondeg}
In the special case where we have the non-degeneracy condition 
\eql{\label{eq-hyp-ri}
	\usol(\porig) = -\nabla F(\xsol(\porig),\porig)\in \ri(\partial R(\xsol(\porig))),
}
our result simplifies and we recover the known exact identification result.
To see this, we note that \eqref{eq-hyp-ri} also reads $\usol(\porig)\in \Jj_R(\ens{\xsol(\porig)})$, and thus $\Man^*_{\usol(\porig)} =\Jj_R(\ens{\xsol(\porig)})$. In this case, \eqref{eq-ident-manif-hybrid}~becomes $\Man_{\xsol(p)}=\Man_{\xorig}$ for all $p$ close to $\porig$. Such a result was also established in \cite{HareLewis04,lewis2013partial} under condition \eqref{eq-hyp-ri}, when $F(\cdot,\porig)$ is also locally $C^2$ around $\xsol(\porig)$ and $R$ is partly smooth at $\xsol(\porig)$ relative to a $C^2$-smooth manifold $\Man_{\xsol(\porig)}$. \hl{An example of this non-degenerate scenario is shown in Figure~\ref{fig:projsens}(b).}
% On va simplifier un peu et ne pas parler de ce detail
%\footnote{The results in \cite{HareLewis04,lewis2013partial} do not even require convexity, but only prox-regularity.}. 
Our result covers the more delicate degenerate situation where $\usol(\porig)$ might be on the relative boundary of the subdifferential, but requires the stronger mirror-stratifiability structure on the non-smooth part. However, the active set at $\xsol(p)$ is in general not unique for all $p$ close to $\porig$, \hl{hence the terminology enlarged activity.}
\end{rem}

\begin{exmp}
\hl{The result of Theorem~\ref{thm-stability-ps} is illustrated in Figure~\ref{fig:projsens}, for the projection problem on a $\ell_\infty$ ball, in both the degenerate and the non-degenerate situations. More precisely, we take $E(x,p)=\tfrac{1}{2}\norm{x-p}^2+\iota_{\Ball_{\ell_\infty}}(x)$, so that $\usol(p) = p - \xsol(p)$, where $\xsol(p) = \Proj_{\Ball_{\ell_\infty}}(p)$. Figure~\ref{fig:projsens}(a) illustrates the degenerate case where $\usol(\porig)$ belongs to the relative boundary of the normal cone of $\Ball_{\ell_\infty}$ at $\xsol(\porig)$. Taking a perturbation $p$ around $\porig$ entails that $\xsol(p)$ belongs either to $\Man_1 = \Man_{\xsol(\porig)}$ or to the enlarged stratum $M_2 = \Jj_{R^*}(\Man^*_{\usol(\porig)}) = \Jj_{R^*}(\Man^*_2)$. In the non-degenerate case of Figure~\ref{fig:projsens}(b) presented in Remark~\ref{rem:compnondeg}, the optimal solution $\xsol(p)$ is always on $\Man_1$ for small perturbations.}
\end{exmp}

%
% On va simplifier un peu et ne pas parler du resultat ci-dessous
%
%\begin{rem}[Uniqueness]
%%	\todo{Jalal, could you provide some conditions based on partial smoothness? Or other?}
%A sufficient condition for $\xsol(\porig)$ to be the unique minimizer of $E(\cdot,\porig)$ is that
%\[
%R'(\xsol(\porig);h) > \dotp{\usol(\porig)}{h}, \quad \forall h \neq 0 ,
%\]
%where $R'$ is the directional derivative of $R$. 
%
%\todo{JF: A necessary and sufficient uniqueness condition is that the descent cone of $E(\cdot,\porig)$ is trivial. However, this does not say that much in this case as there is no additional structure in the composite problem. We can also provide yet another sufficient uniqueness conditions (that implies the one above) if we additionally assume that $F(\cdot,\porig)$ is locally $C^2$ around $\xsol(\porig)$, that \eqref{eq-hyp-ri} holds and that $\nabla^2 F(\xsol(\porig),\porig) \cap \lin(\partial R(\xsol(\porig)))^\perp=\ens{0}$ (this is \cite[Proposition~12(i)]{2015-Liang-InterialFB}). But this involves \eqref{eq-hyp-ri} which trivializes our theorem. I do not think we should go further than this because of these reasons.}
%\end{rem}

\begin{rem}[Quadratic growth]
%	\todo{G: Discuss Lipschitz regularity of $\xsol(p)$. }
We can also establish the Lipschitz continuity of $\xsol(p)$ under the conditions of Theorem~\ref{thm-stability-ps}, under an additional second-order growth condition, just as in classical sensitivity analysis (see~\cite[Section~4.2.1]{bonnans2000perturbation}). Let us assume that there exists a neighbourhood $\Vv$ of $\xsol(\porig)$ and $\kappa>0$ such that
	\eql{\label{eq-grosse}
		E(x,\porig) \geq E(\xsol(\porig),\porig) +\kappa\norm{x-\xsol(\porig)}^2 \qquad \forall x \in \Vv.
	}
Since $\xsol_k \to \xsol(\porig)$ for $p_k \to \porig$ (see the proof of Theorem~\ref{thm-stability-ps}), we have $\xsol_k \in \Vv$ for $k$ large enough. If, moreover, $\nabla F(x,\cdot)$ is Lipschitz-continuous in a neighbourhood of $\porig$ with Lipschitz constant $\nu$ independent of $x \in \Vv$, we get
\begin{equation*}
\begin{aligned}
\kappa\norm{\xsol_k-\xsol(\porig)}^2 
&\leq E(\xsol_k,\porig) - E(\xsol(\porig),\porig) \\
&= \bpa{(E(\xsol(\porig),p_k) - E(\xsol_k,p_k)) - (E(\xsol(\porig),\porig) - E(\xsol_k,\porig))} \\
&\qquad - \bpa{E(\xsol(\porig),p_k) - E(\xsol_k,p_k)} \\
\text{\scriptsize{($\xsol_k$ minimizes $E(\cdot,p_k)$)}} 
&\leq (E(\xsol(\porig),p_k) - E(\xsol_k,p_k)) - (E(\xsol(\porig),\porig) - E(\xsol_k,\porig) \\
\text{\scriptsize{(By \eqref{eq-primal})}} 
&= (F(\xsol(\porig),p_k) - F(\xsol_k,p_k)) - (F(\xsol(\porig),\porig) - F(\xsol_k,\porig)) \\
\text{\scriptsize{(Mean-Value Theorem)}} 
&\leq \bpa{\sup_{x \in \Vv} \norm{\nabla F(x,p_k) - \nabla F(x,\porig)}}\norm{\xsol_k - \xsol(\porig)} \\
&\leq \nu \norm{p_k-\porig}\norm{\xsol_k - \xsol(\porig)} .
\end{aligned}
\end{equation*}
whence we conclude that
\[
d(\xsol(\porig),\Argmin E(\cdot,p_k)) \leq \nu/\kappa \norm{p_k-\porig} .
\]
\end{rem}

%% file: sections/inverse-problems.tex
% !TEX root = ../StratificationSensitivity.tex

%%%%%%%%%%%%%%%%%%%%%%%%%%%%%%%%%%%%%%%%%%%%%%%%%%%%%%%%%%%%
\section{Regularized Inverse Problems}
%%%%%%%%%%%%%%%%%%%%%%%%%%%%%%%%%%%%%%%%%%%%%%%%%%%%%%%%%%%%
\label{sec-ip-regul}

A typical context where our framework of mirror-stratifiability is of usefulness is that of studying stability to noise of regularized linear inverse problems. We come back to the situation presented in introduction: studying stability issues to perturbed observations of the form $y=\yorig+w$ amounts to analyzing sensitivity of the minimizers and the optimal value function in~\eqref{eq-variational-ip-intro-noisy} when the parameter $p=(\la,y)$ evolves around the reference point $\porig=(0,\yorig)$. 
Unlike the usual sensitivity analysis theory setting recalled in the previous section (e.g.~\cite{bonnans2000perturbation}), here the objective $E$ may not even be continuous at $\porig$. 

We provide hereafter some pointers to the relevant literature, then we develop our sensitivity results for mirror-stratifiable functions. These results involve primal and dual solutions to the noiseless problem 
\begin{align}
	\label{eq-variational-ip-intro-noiseless}\tag{$\Pp(0,y)$}
	&\umin{x \in \RR^N} E(x,(0,y)) \eqdef R(x) \quad\text{s.t.}\quad \Phi x  =  y, 
\end{align}
%
%
%\begin{align}\label{eq-variational-ip-intro-noisy}\tag{$\Pp(\la,y)$}
%	&\umin{x \in \RR^N}  E(x,(\la,y)) \eqdef R(x) + \frac{1}{2\la}\norm{ y - \Phi x }^2, 
%	&\text{for } \la>0, \\
%	\label{eq-variational-ip-intro-noiseless}\tag{$\Pp(0,y)$}
%	&\umin{x \in \RR^N} E(x,(0,y)) \eqdef R(x) \quad\text{s.t.}\quad \Phi x  =  y, 
%	&\text{for } \la=0.
%\end{align}
%The parameter $\la>0$ controls the regularization strength and should be adapted to the perturbation level $\norm{w}$. When there is no noise ($w=0$), one sets $\la=0$ and solves~\eqref{eq-variational-ip-intro-noiseless} with $y=\yorig$. 
%
%%The foundations of regularization theory can be traced back to the pioneering work of the Russian school, and in particular of Tikhonov when he proposed the notion of conditional well-posedness~\cite{tikhonov1977solutions} .

%%%%%%%%%%%%%%%%%%%%%%%%%%%%%%%%%%%%%%%%%%%%%%%%%%%%%%%%%%%%%%%%%%%%%%%%%%%%%%%
\subsection{Existing sensitivity results for regularized Inverse Problems}
\label{sec-pw-ip-regul}

\paragraph{Lipschitz stability}

Assume there exists a dual multiplier $\eta \in \RR^P$ (sometimes referred to as a ``dual certificate'') for the noiseless constrained problem~\eqref{eq-variational-ip-intro-noiseless} taken at $y=\yorig$ such that $\Phi^* \eta \in \partial R(\xorig)$. The latter condition is equivalent to $\xorig$ being a minimizer of~\eqref{eq-variational-ip-intro-noiseless}. This condition goes by the name of the ``source'' or ``range'' condition in the inverse problems literature. It has been widely used to derive stability results in terms of $\norm{\Phi \xsol(\la,y)-\Phi\xorig}$ or $R(\xsol(\la,y))-R(\xorig)-\dotp{\Phi^* \eta}{\xsol(\la,y)-\xorig}$;
%(the Bregman divergence associated to $R$) with $\lambda=c_0\norm{w}$; 
 see~\cite{scherzer2009variational} and references therein. 
%It was considered in~\cite{burger2013converg}, and it only implies a weak stability (according to a Bregman divergence associated to $R$) to noise $w$ in the observations.
%
To afford stability in terms of $\norm{\xsol(\la,y)-\xorig}$ directly, the range condition has to be strengthened to its non-degenerate version $\Phi^* \eta \in \ri(\partial R(\xorig))$. 
%which is similar to the non-degeneracy condition~\eqref{eq-ri-condition} imposed when analyzing partly smooth problems (but recall that problem~\eqref{eq-variational-ip-intro-noiseless} is not partly-smooth when $\la=0$). 
It has been shown that this condition implies that the set-valued map $(\la,y) \mapsto \xsol(\la,y)$ is Lipschitz-continuous at $(0,\yorig)$; see~\cite{grasmair2011necessary}
% in the case of $R=\norm{\cdot}_1$, 
and \cite{2014-vaiter-ps-review}.
% for any finite-valued convex regularizer $R$.
 
%%
\paragraph{Active set stability}

In the case where $R$ is partly smooth, one can approach an even more complete sensitivity theory by studying stability of the partly smooth manifold of $R$ at $\xorig$. In particular, it can be shown from that if \hl{an appropriate non-degeneracy assumption holds, see~\cite{2014-vaiter-ps-consistency}},
%$\Phi^* \etaorig \in \ri(\partial R(\xorig))$, where $\etaorig$ is the minimum $\ell_2$-norm dual certificate, and with the provisio that $\norm{w}/\la$ and $\la$ are small enough (i.e. the noise is not too large), 
then problems~\eqref{eq-variational-ip-intro-noiseless} and~\eqref{eq-variational-ip-intro-noisy} have unique minimizers (respectively $\xorig$ and $\xsol(\la,y)$), and $\xsol(\la,y)$ lies on the partly smooth manifold of $R$ at~$\xorig$. Observe that compared to Lipschitz stability, active set stability is more demanding as the non-degeneracy condition has to hold for a specific dual multiplier, which is obviously more stringent. 
This type of results has appeared many times in the literature for special cases, e.g. for the $\ell_1$ norm~\cite{fuchs2004on-sp,zhao2006model}, %the total variation regularizer~\cite{vaiter-analysis}, 
%polyhedral regularizers~\cite{vaiter13polyhedral}, 
%the $\ell_{1,2}$ group Lasso~\cite{bach2008group}, 
the nuclear norm~\cite{bach2008trace}.
% and geometrically decomposable regualrizers~\cite{lee2013consistency}. 
The work in~\cite{vaiter-model-linear15,2014-vaiter-ps-consistency} has unified all these results.
% under the umbrella of partial smoothness.

%%
\paragraph{Non-degeneracy in practice for deconvolution and compressed sensing}

The above results require that some abstract non-degeneracy condition holds, which imposes strict limitations on practical situations. In particular, when $\Phi$ is a convolution operator and $\xorig$ is sparse, \cite{candes2013super} studies Lipschitz stability and~\cite{duval2013spike} support stability. In this setting, the non-degeneracy condition holds whenever the \hl{non-zero entries} are separated enough, which is not often verified. 
Another setting where this stability theory has been applied in when $\Phi$ is drawn from a random matrix ensemble, i.e.\;compressed sensing. For a variety of partly smooth regularizers (including the $\ell_1$, nuclear and $\ell_{1,2}$ norms), the non-degeneracy condition holds with high probability if, roughly speaking, the sample size $P$ is sufficiently larger than the ``dimension'' of the active set at $\xorig$ (see e.g.\,\cite{candes2005decoding, dossal2012sharp}). Again this is a clear limitation as illustrated in Section~\ref{sec-motivnum}.
\subsection{Primal and dual problems}%{Inverse problems in imaging and learning}

Suppose we have observations of the form \eqref{eq-fwd-ip}, and we want to recover $\xorig$ (or a provably good approximation of it). As advocated in Section~\ref{sec-motivation}, a popular approach is to adopt a regularization framework which can be cast as the optimization problem~\eqref{eq-variational-ip-intro-noisy} (for $\la>0$) and~\eqref{eq-variational-ip-intro-noiseless} (when $\la=0$). In the sequel, \hl{we assume that
\begin{equation}\label{eq:pbinvcoer}
  \Rinf(z) > 0, \quad \forall z \in \ker(\Phi) \setminus \ens{0} ,
\end{equation}
where $\Rinf$ is the asymptotic (or recession) function of $R$, defined as 
\begin{equation*}
\Rinf(z) \eqdef \lim_{t\to+\infty}\frac{R(x+tz)-R(x)}{t},  \quad  \forall x \in \dom(R) .
\end{equation*}
Condition \eqref{eq:pbinvcoer} is a necessary and sufficient condition for the set of minimizers of \eqref{eq-variational-ip-intro-noisy} and~\eqref{eq-variational-ip-intro-noiseless} to be non-empty and compact~\cite[Lemma~5.1]{vaiterphd}. It is satisfied for example when $R$ is coercive.} 
%
%In the setting of this section, it can be easily shown, using \cite[Proposition~3.1.3]{Auslender03}, that the set of minimizers of \eqref{eq-primal} (hence of \eqref{eq-variational-ip-intro-noisy} and~\eqref{eq-variational-ip-intro-noiseless}) is non-empty and compact if and only if
%\begin{equation}\label{eq:pbinvcoer}
%  \ker(\Rinf)\cap \ker \Phi=\{0\} . % \tag{$H_{(Coer)}$}
%\end{equation}
%This is also a sufficient and necessary condition for the objective in $E$ to be level-bounded in $x$ uniformly in $p$.

\begin{rem}[Discontinuity of $F$]
Letting $p \eqdef (\la,y) \in \Param \eqdef \RR_+ \times \RR^P$, we see that~\eqref{eq-variational-ip-intro-noisy} and~\eqref{eq-variational-ip-intro-noiseless} are instances of~\eqref{eq-primal}, by setting
\eql{\label{eq-ip-setup} 
	F(x,p) \eqdef
	\choice{
		\frac{1}{2\la}\norm{ y - \Phi x }^2 &\qifq \la >0, \\
		\ind_{\Hh_y}(x) 					&\qifq \la=0, 
	}
	\qwhereq
	\Hh_y \eqdef \enscond{x \in \RR^N}{\Phi x = y}.
}
%The goal in this section is to study the properties of solutions to~\eqref{eq-primal} when $\yorig$ undergoes small perturbations (i.e. the noise $w$ is small) and $\la$ is small, i.e. for $p=(\la,y)$ close to $\porig \eqdef (0,\yorig)$. 
%It is important to realize that, while we made the identification with the objective in \eqref{eq-primal}, 
The corresponding function $F$ considered does not obey the assumptions of Theorem~\ref{thm-stability-ps}.
Indeed, the parameter set $\Param$ is not open, with $\porig=(0,\yorig)$ that lives on the boundary of $\Param$, and $F$ is only lsc at such $\porig$ (because of the affine constraint $\Phi x=y$). 
%This is the reason behind introducing a specific dual vector $\uorig$ associated to $\xorig$ (which should be contrasted to the simpler choice $\usol(\porig)=-\nabla F(\xsol(\porig),\porig)$ made in Theorem~\ref{thm-stability-ps}).
%
The main consequence will be that one \hl{can no longer} allow $p$ to vary freely nearby $\porig$. 
%Rather, as we detail next in Theorem~\ref{thm-ip}, $p$ must be restricted to the cone portion~\eqref{eq-domain-ip} on which a localization result can be stated. 
\end{rem}

%\begin{rem}[General loss/data fidelity]
%%\todo{G: explain that one can replace $\norm{\Phi x-y}^2$ by a more general function of the form $g(\Phi x,y)$, where $g$ should satisfy $C^2$, strong convexity, etc. Cite our paper with Samuel. }
%It is possible to extend Theorem~\ref{thm-ip} below to account for general loss/data fidelity terms beyond the quadratic one, i.e. $\frac{1}{2}\norm{ y - \Phi x }^2$. More precisely, one can consider losses of the form $G(\Phi x,y)$, where $G(\cdot,y)$ is a $C^2$ strongly convex function uniformly in $y$, and $\nabla F(\cdot,y)$ is Lipschiz continuous uniformly in $y$; see~\cite{2014-vaiter-ps-consistency}.
%\end{rem}
%
%
%\begin{rem}[Identifiability] Classical theoretical studies of stability of regularized inverse problems assume that $(\Pp(0,\yorig))$ has a unique solution $\xsol(\porig)$ and moreover $\xsol(\porig)=\xorig$, a condition often referred to as ``identifiability'' of $\xorig$. Our analysis does not require this either, and we state our stability result using $\xsol(\porig)$, without any explicit reference to the original vector $\xorig$. 
%%Our results are thus intended to be used as a refinement of classical results: once one knows that  $\xorig=\xorig$, one can refine the identifiability of the vector by providing a manifold identification localization. 
%\end{rem}

%%%%%%%%%%%%%%%%%%%%%%%%%%%%%%%%%%%%%%%%%%%%%%%%%%%%%%%%%%%%

In order to study the sensitivity of solutions, we look at the Fenchel-Rockafellar dual problem, which reads (for all $\la \geq 0$)
\eql{\label{eq-dual-ip}\tag{$\Dd(\la,y)$}
	\umax{q \in \RR^P} \dotp{q}{y} - \frac{\la}{2} \norm{q}^2 - R^*(\Phi^* q). 
}
We denote $\Dsol(\la,y)$ the set of solutions of $(\Dd(\la,y))$. Note that for $\la>0$, thanks to strong concavity, there is a unique dual solution $\qsol(\la,y)$, i.e.~$\Dsol(\la,y)=\ens{\qsol(\la,y)}$. 
%Moreover, strong duality holds between~\eqref{eq-primal} and~\eqref{eq-dual-ip} since $R$ is proper. In turn, 
We also have from the primal-dual extremality relationship that for any primal solution $\xsol(\la,y)$ of~\eqref{eq-variational-ip-intro-noisy}, 
\eql{\label{eq:extrprimaldual}
	\qsol(p) = \frac{y-\Phi \xsol(p)}{\la} \qandq \Phi^* \qsol(p) \in \partial R(\xsol(p)).
}
While $(\Dd(0,y))$ is the dual of $(\Pp(0,y))$, it is important to realize that it is not the limit of $(\Dd(\la,y))$ in the sense that its set of dual solutions is in general not a singleton. 
Lemma~\ref{lemma-conv-dual-ip} hereafter singles out a specific dual optimal solution (sometimes called ``minimum norm certificate'') defined by
%\eql{\label{eq-dfn-minnorm-cert}
%	\qF(0,y) \eqdef \uargmin{q \in \RR^P} \enscond{\norm{q}_2}{ q \in \Dsol(0,y) }.
%}
%\end{defn}
%
%Using \eqref{eq:extrprimaldual} $\xsol(0,y)$ of $(\Pp(0,y))$, one can conveniently re-write this definition as
\eql{\label{eq-dfn-minnorm-cert}
	\qF(0,y) = \uargmin{q \in \RR^P} \enscond{\norm{q}\!}{ \Phi^* q \in \partial R(\xsol(0,y)) } = \uargmin{q \in \RR^P} \enscond{\norm{q}\!}{ q \in \Dsol(0,y) }.
}
The following two important lemmas ensure the convergence of the solutions to the primal and dual problems as $\la \rightarrow 0$ and as $\norm{y-\yorig} \to 0$. 

\begin{lem}[Primal solution convergence]\label{lemma-conv-primal-ip}
	Assume that $\xorig$ is the unique solution to $(\Pp(0,\yorig))$.
	For any sequence of parameters $p_k = (\la_k,y_k)$ with $\la_k>0$ such that 
	\eq{
		\pa{\frac{\norm{y_k-\yorig}^2}{\la_k},\la_k} \longrightarrow (0,0), 
	}
	and any solution $\xsol(p_k)$ of $(\Pp(p_k))$, 
	we have $\xsol(p_k) \rightarrow \xorig$.
\end{lem}

\begin{proof}
Denote $\xsol_k \eqdef \xsol(p_k)$. Since $\pa{\xsol_k}_{k \in \NN}$ is a bounded sequence \hl{by \eqref{eq:pbinvcoer}}, one can extract a converging subsequence, \hl{which for simplicity, we denote again} $\xsol_k \rightarrow \xlim$.  
Optimality of $\xsol_k$ implies that 
\eql{\label{eq-proof-ip}
R(\xsol_k)	\leq \frac{1}{2\la_k} \norm{y_k - \Phi \xsol_k}^2 + R(\xsol_k)
	\leq 
	\frac{1}{2\la_k} \norm{y_k - \yorig}^2 + R(\xorig).
}
Passing to the limit and using the hypothesis that $\frac{1}{\la_k} \norm{y_k - \yorig}^2 \rightarrow 0$, we get
\[
\limsup_k R(\xsol_k) \leq R(\xorig).
\]
On the other hand, lower semi-continuity of $R$ entails $
\liminf_k R(\xsol_k) \geq R(\xlim).
$
Combining these two inequalities, we deduce that
$
R(\xlim) \leq R(\xorig) .
$
Since $R$ is proper and lsc, it is bounded from below on bounded sets~\cite[Corollary~1.10]{rockafellar1998var}. Let $\ul{r}=\inf_{k} R(\xsol_k)$ which then satisfies $-\infty < \ul{r} < +\infty$. 
Substracting $\ul{r}$ from~\eqref{eq-proof-ip} and multiplying by $\la_k$, one obtains
\eql{\label{eq-proof-ip2}
\frac{1}{2} \norm{y_k - \Phi \xsol_k}^2 \leq \frac{1}{2} \norm{y_k - \Phi \xsol_k}^2 + \la_k (R(\xsol_k)-\ul{r})
	\leq 
	\frac{1}{2} \norm{y_k - \yorig}^2 + \la_k (R(\xorig)-\ul{r}).
}
%Since~$\la_k  \to 0$, $R$ is lsc, and $\pa{R(\xsol_k)-\ul{r}}_{k \in \NN}$ is a non-negative and bounded sequence, we have
%	\begin{align*}
%	\limsup_{k} \pa{\la_k (R(\xsol_k)-\ul{r})} &\leq \lim_{k}\la_k  \limsup_{k}(R(\xsol_k)-\ul{r}) = 0 \qandq \\
%	\liminf_{k} \pa{\la_k (R(\xsol_k)-\ul{r})} &\geq \lim_{k}\la_k  \liminf_{k}(R(\xsol_k)-\ul{r}) \\
%	&\geq  (R(\xlim)-\ul{r})\lim_{k}\la_k = 0 ,
%	\end{align*}
%and thus $\lim_{k \to \infty}\pa{\la_k (R(\xsol_k)-\ul{r})} = 0$.
Consequently, passing to the limit in~\eqref{eq-proof-ip2} shows that $\norm{\yorig-\Phi\xlim} = 0$, i.e. $\xlim$ is a feasible point of $(\Pp(0,\yorig))$.
Altogether, this shows that $\xlim$ is a solution of $(\Pp(0,\yorig))$, and by uniqueness of the minimizer, $\xlim=\xorig$.	
\end{proof}

\begin{lem}[Dual solution convergence]\label{lemma-conv-dual-ip}
	For any sequence of parameters $p_k = (\la_k,y_k)$ such that 
	\eq{
		\pa{\frac{\norm{y_k-\yorig}}{\la_k},\la_k} \longrightarrow (0,0), 
	}
	we have $\qsol(p_k) \rightarrow \qF(0,\yorig)$.
\end{lem}

\begin{proof}
%We show the convergence in two steps:
%	 (i) $\qsol(\la,y_k) \rightarrow \qsol(\la,\yorig)$ for any fixed $\la>0$ ; 
%	 (ii) $\qsol(\la_k,\yorig) \rightarrow \qF(0,\yorig)$.	 
%%%%
%As far as (i) is concerned, we notice that $\qsol(\la,y)=\prox_{R^*\circ\Phi^*/\la}(y/\la)$.
%Thus, by virtue of \cite[Theorem~2.26]{rockafellar1998var}, we conclude that $\qsol(\la,y_k) \to \qsol(\la,\yorig)$.
%Using the fact the latter is non-expansive (see e.g.~\cite[thm. XV.4.1.4]{LeMarechalJBHU}), one thus obtains $\norm{\qsol(\la,y_k) - \qsol(\la,\yorig)} \leq \frac{y_k-\yorig}{\la}$, which converges to zero by hypothesis.
By the triangle inequality, we have
\[
\norm{\qsol(\la_k,y_k) - \qF(0,\yorig)} \leq \norm{\qsol(\la_k,y_k) - \qsol(\la_k,\yorig)} + \norm{\qsol(\la_k,\yorig) - \qF(0,\yorig)} .
\]
For the first term, we notice that $\qsol(\la,y)=\prox_{R^*\circ\Phi^*/\la}(y/\la)$ for $\la > 0$. Thus, Lipschitz continuity of the proximal mapping entails that
\[
\norm{\qsol(\la_k,y_k) - \qsol(\la_k,\yorig)} \leq \frac{\norm{y_k-\yorig}}{\la_k},
\]
which in turn shows that $\qsol(\la_k,y_k) \to \qsol(\la_k,\yorig)$.
%%%
Let us now turn to the second term. Using the respective optimality of $\qsol_k=\qsol(\la_k,\yorig)$ and $\qF \eqdef \qF(0,\yorig) \in \Dsol(0,\yorig)$, one obtains
\begin{equation}\label{eq-proof-dual-1}
\begin{aligned}
	-\dotp{\qsol_k}{\yorig} + \frac{\la_k}{2} \norm{\qsol_k}^2 + R^*(\Phi^* \qsol_k) &\leq
	-\dotp{\qF}{\yorig} + \frac{\la_k}{2} \norm{\qF}^2 + R^*(\Phi^* \qF) \\
	&\leq -\dotp{\qsol_k}{\yorig} + \frac{\la_k}{2} \norm{\qF}^2 + R^*(\Phi^* \qsol_k) ,
%	-\dotp{\qF}{\yorig} + R^*(\Phi^* \qF) &\leq
%	-\dotp{\qsol_k}{\yorig} + R^*(\Phi^* \qsol_k).
\end{aligned}		
\end{equation} 
whence we get $\norm{\qsol_k} \leq \norm{\qF}$, which shows in particular that $\pa{\qsol_k}_{k \in \NN}$ is a bounded sequence. We can thus extract any converging subsequence, \hl{ which for simplicity, we denote again} $\qsol_k \rightarrow \qlim$. Passing to the limit in~\eqref{eq-proof-dual-1}, using the fact that $R^*$ is lsc, one obtains
\eq{
-\dotp{\qlim}{\yorig} + R^*(\Phi^* \qlim) \leq -\dotp{\qlim}{\yorig} + \liminf_k R^*(\Phi^* \qsol_k) \leq -\dotp{\qF}{\yorig} + R^*(\Phi^* \qF), 
}
which shows that $\qlim \in \Dsol(0,\yorig)$. This together with the fact that $\norm{\qsol_k} \leq \norm{\qF}$ as we already saw, shows that $\qlim = \qF$ by uniqueness of $\qF$ in \eqref{eq-dfn-minnorm-cert}. %Using the triangle inequality and passing to the limit we conclude.
\end{proof}

%%%%%%%%%%%%%%%%%%%%%%%%%%%%%%%%%%%%%%%%%%%%%%%%%%%%%%%%%%%%
\subsection{Sensitivity} 
\label{sec-ip-regul-sensitiv}

We are now in position to state the main result of this section, which tracks the strata of $\xsol(p)$ in the regime where the perturbation $\norm{w}=\norm{y-\yorig}$ is sufficiently small. 

\begin{thm}\label{thm-ip}
	Suppose that $\xorig$ is the unique solution to $(\Pp(0,\yorig))$.
	Assume furthermore that $R$ is mirror-stratifiable with respect to the primal-dual stratifications $(\Strat,\Strat^*)$.
	\hl{If there are} constants $(C_0,C_1)$ depending only on $\xorig$ such that for all $p$ in 
	\eql{\label{eq-domain-ip}
		\enscond{p=(\la,y)}{C_0 \norm{y-\yorig} \leq \la \leq C_1 },  
	}
	\hl{then there} exists a minimizer $\xsol(p)$ of $E(\cdot,p)$ localized as follows
	\eql{\label{eq-ident-manif}
		\Man_{\xorig} \leq \Man_{\xsol(p)} \leq {\Jj_{R^*}}(\Man^*_{\uF})
		\qwhereq
		\uF \eqdef \Phi^* \qF(0,\yorig).
	}
\end{thm}

\begin{proof}
Under \eqref{eq-domain-ip} with for instance $C_1$ small enough, \hl{one can easly check} that there exists $k$ large enough such that the regime required for $(y_k,\la_k)$ to apply Lemma~\ref{lemma-conv-primal-ip}~and~\ref{lemma-conv-dual-ip} is attained.
%, \todo{maybe give more details and check this carefully}
In turn, we have a converging primal-dual pair $(\xsol(p_k),\Phi^*\qsol(p_k)) \rightarrow (\xorig,\uF)$. One can then apply Theorem~\ref{thm:main} to conclude.
\end{proof}

%\begin{rem}[Exact identifiability]
%\todo{G: explain the special case $\Phi^* \qF(0,\yorig) \in \ri(\partial R(\xorig))$, links to previous works section (``irrepresentability condition''), then one can use of partial smoothness to state smoothness of the map $p \rightarrow \xsol(p)$. One can then check the condition $\Phi^* \qF(0,\yorig) \in \ri(\partial R(\xorig))$ by replacing $\Phi^* \qF(0,\yorig)$ by the so-called linearized pre-certificate. }
%
%\todo{JF: One can take a few sentences from our review chapter.}
%\end{rem}

%% file: sections/algorithms.tex
% !TEX root = ../StratificationSensitivity.tex

\section{Activity \hl{Localization} with Proximal Splitting Algorithms}
\label{sec-algorithm}

Proximal splitting methods are algorithms designed to solve large-scale structured optimization and monotone inclusion problems, by evaluating various first-order quantities such as gradients, proximity operators, linear operators, all separately at various points in the course of \hl{an} iteration. Though they can show slow convergence
% rates than other competitors (such as second-order methods or interior point algorithms for conic programming), 
each iteration has a cheap cost. We refer to e.g.~\cite{bauschke2011convex,beck2009gradient,combettes2011proximal,parikh2013proximal} for a comprehensive treatment and review. 

Capitalizing on our \hl{enlarged activity identification} result for mirror-stratifiable functions, we now instantiate its consequences on \hl{finite activity localization} of proximal splitting algorithms. While existing results on finite identification \hl{(of a single active set)} strongly rely on partial smoothness around a non-degenerate cluster point~\cite{HareFB11,2015-Liang-InterialFB,2016-Liang-DR,LiangPhD16}, we examine here intricate situations where neither of these assumptions holds.

%
%The essence of partial smoothness is in the fact that algorithms generating iterates converging to a non-degenerate cluster point, along with approximate minimality certificates, will identify the active/partly smooth manifold associated to the cluster point in finitely many iterations~\cite{HareLewis04}. 
%
%In turn, this allows a much refined local convergence analysis of most proximal splitting schemes~\cite{HareFB11,2015-Liang-InterialFB,2016-Liang-DR,Liang14,Liang16,LiangPhD16}. It also opens the door to higher order acceleration strategies~\cite{Lemarechal-ULagrangian,miller2005newton}.

%
%This analysis is made possible thanks to the fact that the proximal mapping is able to track the identifiable manifold as shown by~\cite{HareLewis04,HareFB11}.  
%
%Once the manifold is identified by these first order scheme, it is possible to show that they enjoy a linear convergence rate It is also possible to apply higher order methods such as a Riemannian Newton scheme on the identified manifold to further speed up the convergence.

%\todo{G: this is in part copy-paste from Jingwei's articles, most likely needs rewriting. }

%%%%%%%%%%%%%%%%%%%%%%%%%%%%%%%%%%%%%%%%%%%%%%%%
\subsection{Forward-Backward algorithm}
\label{sec-fb}

%\todo{G: detail a bit more, ref to previous works section.}

The Forward--Backward (FB) splitting method \cite{lions1979splitting} is probably the most well-known proximal splitting algorithm. 
In our context, it can be used to solve optimization problems with the additive ``smooth + non-smooth'' structure of the form
%Consider the composite convex optimization problem
\begin{equation}\label{eq:pbFB}
	\min_{x\in \RR^N} f(x) + R(x) ,
\end{equation}
where $f \in C^1(\RR^N)$ is convex with $L$-Lipschitz gradient, and $R$ is a proper lsc and convex function. We assume that $\Argmin(f+R) \neq \emptyset$.
%% FB %%
%The standard (non-relaxed) form of FB is instantiated in~\eqref{eq-fb-iter-intro} for the case of regularized inverse problems. 
The FB iteration in relaxed form reads~\cite{combettes2004solving} 
\begin{equation}\label{eq:FB}
	\Iter{x} = (1-\tau_k) \iter{x} + \tau_k\prox_{\ga_k R}(\iter{x}-\ga_k\nabla f(\iter{x})) ,
\end{equation}
with $\ga_k \in ]0,2/L[$ and $\tau_k \in ]0,1]$, where $\prox_{\ga R}$, $\ga > 0$, is the proximal mapping of $R$,
\eql{\label{eq:prox}
\Prox_{\gamma R}(x) \eqdef \uargmin{z\in\RR^N} \frac{1}{2}\norm{z-x}^2 + \gamma R(z)  .
}
%With such a choice, it is well-known that $\iter{x}$ converges to a point in $\Argmin(f+R)$, see~e.g.\cite[Theorem~3.4]{}.
%\todo{G: correct stated like this?}.
%
%We recall that the proximal operator of a convex function $F$ is defined by 
%\begin{equation}\label{eq:prox}
%	\prox_F(u) \eqdef \argmin_{y} F(y) + \frac{1}{2}\norm{y-x}^2.
%\end{equation}
Different variants of FB method were studied, and a popular trend is the inertial schemes which aim at speeding up the convergence (see \cite{Nesterov83,beck2009fista}, and the sequence-convergence version as proved recently in \cite{chambolle2015convergence}).
%

%\subsection{Existing finite identification results}.
%When the non-smooth part of the structured optimization problem is partly smooth, 
Under the non-degeneracy assumption $-\nabla f(\xsol) \in \ri(\partial R(\xsol))$, it was shown in \cite{2015-Liang-InterialFB} that FB and its inertial variants correctly identify the active manifold in finitely many iterations, and then enter a local linear convergence regime. These results encompass many special cases such as those studied in~\cite{hale07,bredies2008linear,tao2015local}. 
\hl{Beyond this non-degenerate case, we establish now the general localization of active strata}.

\begin{thm}\label{thm-FB}
	Consider the FB iteration \eqref{eq:FB} to solve \eqref{eq:pbFB} with $0 < \inf_k \ga_k \leq \sup_k \ga_k < 2/L$ and $\inf_k \tau_k > 0$. Then $\iter{x}$ converges to $\xsol \in \Argmin(f+R)$. Assume that $R$ is also mirror-stratifiable with respect to $(\Strat,\Strat^*)$, then for $k$ large enough,
	\eq{
		\Man_{\xsol} \leq \Man_{\iter{x}} \leq \Jj_{R^*}(\Man^{*}_{\usol})
		\qwhereq
		\usol \eqdef -\nabla f(\xsol).
	}
\end{thm}

\begin{proof}
Convergence of the sequence $\pa{\iter{x}}_{k \in \NN}$ to $\xsol$ is obtained from~\cite[Corollary~6.5]{combettes2004solving}.
Moreover, since the proximal mapping is the resolvent of the subdifferential, \eqref{eq:FB} is equivalent to the monotone inclusion
\[
\Iter{u} \eqdef \frac{\iter{x} - \Iter{v}}{\ga_k} - \nabla f(\iter{x}) \in \partial R(\Iter{v}),
\]
where $\Iter{v} \eqdef \frac{\Iter{x}-\iter{x}}{\tau_k}+\iter{x}$. In turn, with the conditions $\inf_k \tau_k > 0$ and $\inf_k \ga_k > 0$, and continuity of $\nabla f$, we have $\iter{v} \to \xsol$ and thus $\iter{u} \to -\nabla f(\xsol) \in \partial R(\xsol)$. It then remains to apply Theorem~\ref{thm:main} to $(\iter{v},\iter{u})$ and $R$ to conclude.
\end{proof}

%In a nutshell, after finitely many iterations, the sequence of iterates $\iter{x}$ will leave on strata that are precisely localized according to Theorem~\ref{thm-FB}. In contrast, the results of~\cite{2015-Liang-InterialFB} require that $R$ is partly smooth at $\xsol$ relative a smooth manifold $\Man_{\xsol}$, and that $-\nabla f(\xsol) \in \ri(\partial R(\xsol))$. Though they entail the stronger conclusion that $\iter{x} \in \Man_{\xsol}$.
%
%\todo{G: discuss a bit the meaning/relevance of this result.}

%\todo{Ecrire l'identification pour la version acceleree : FISTA}

It can be easily shown that Theorem~\ref{thm-FB} holds for several extensions of the iterate-convergent version of FISTA~\cite{chambolle2015convergence}. We omit the details here for the sake of brevity.

%\begin{rem}
%\begin{itemize}
%%\item Though we did not do it here, Theorem~\ref{thm-FB} remains valid for the inexact version of \eqref{eq:FB}, where summable errors are allowed both in the implicit and explicit steps. This is a clear innovation with respect to existing results on finite identification results under partial smoothness, where such a property is not stable to errors in the implicit step, see~\cite[Section~3.3]{2015-Liang-InterialFB}. 
%\item 
%%\item Theorem~\ref{thm-FB} can also be extended to the case where $f$ is allowed to be non-convex, though a bit more involved arguments will be needed. Indeed, taking $\tau_k \equiv 1$, $f+R$ is lower-bounded and satisfies the Kurdyka--{\L}ojasiewicz inequality, that $0 < \inf_k \ga_k \leq \sup_k \ga_k < 1/L$, one can use \cite[Theorem~5.1]{attouch2013convergence} to deduce that each bounded sequence $\iter{x}$ converges to a {\emph{critical}} point $\xsol$ at which the conclusion of Theorem~\ref{thm-FB} holds.
%\end{itemize}
%\end{rem}

%%%%%%%%%%%%%%%%%%%%%%%%%%%%%%%%%%%%%%%%%%%%%%%%
%\paragraph{Forward-Backward for Regularized Inverse Problems}

We rather take a closer look to the case when the FB scheme~\eqref{eq:FB} to solve~\eqref{eq-variational-ip-intro-noisy} (see also~\eqref{eq-ip-setup}) for $\la > 0$. 
% simplifions un peu en etant un poil imprecis en commentant :
%For this, one observes that~\eqref{eq-variational-ip-intro-noisy} is a special case of~\eqref{eq:pbFB} with $f(x)=\tfrac{1}{2\la}\norm{y-\Phi x}^2$ and $(y,\la)$ are fixed parameters. We assume that~\eqref{eq:pbinvcoer} holds so that the set of minimizers is non-empty and compact.
Putting together Theorem~\ref{thm-ip} and~\ref{thm-FB}, we obtain the following localization result depending only on the data to estimate $\xorig$ assuming that the noise level $\norm{y-\yorig}$ is small enough.
\begin{prop}\label{prop-fb-ip}
	Under the assumptions of Theorem~\ref{thm-ip},
	%Suppose that $(\Pp(0,\yorig))$ has a unique solution $\xorig$. Assume furthermore that $R$ is mirror-stratifiable with respect to the primal-dual stratifications $(\Strat,\Strat^*)$. 
	consider the FB iteration \eqref{eq:FB} to solve \eqref{eq-variational-ip-intro-noisy} with $0 < \inf_k \ga_k \leq \sup_k \ga_k < 2\la/\norm{\Phi}^2$ and $\inf_k \tau_k > 0$.
	%
	% Recall that the iterate $\iter{x}$ of FB defined in~\eqref{eq:FB} converge to some $\xorig$ solution of $(\Pp(y,\la))$. 
	%
%	Then there exists $(C_0,C_1)$ depending only on $\xorig$ such that for any $(\la,y)$ obeying
%	\eql{\label{eq-cone-domain}
%		 C_0 \norm{y-\yorig} \leq \la \leq C_1,  
%	}
	Then, for $k$ large enough, the iterates $\iter{x}$ satisfy 
	\eq{
		M_{\xorig} 
		\leq M_{\iter{x}} 
		\leq \Jj_{R^*}(M_{\uF}^*)
		\qwhereq
		\uF \eqdef \Phi^* q(0,\yorig) ,
	}
	with $q(0,\yorig)$ from~\eqref{eq-dfn-minnorm-cert}.
\end{prop}

\begin{proof}
To lighten notation, denote $p=(\la,y)$. As in Theorem~\ref{thm-FB}, we have $\iter{x} \to \xsol(p)$ a minimizer of \eqref{eq-variational-ip-intro-noisy}. Thus we get
\[
M_{\xorig} \leq M_{\xsol(p)} \leq M_{\iter{x}} ,
\]
where the first inequality comes from Theorem~\ref{thm-ip} and the second one from Theorem~\ref{thm-FB}. This gives the first inequality in the localization result.

Moreover, in the special case at hand $\partial R(\xsol(p)) \ni -\nabla f(\xsol(p))=\Phi^*\qsol(p) \to \uF$ as $p \to (0,\yorig)$, where we invoked~\eqref{eq:extrprimaldual}. It then follows from~\eqref{eq-continuity-inclusion-dual} and the fact that $\Jj_{R^*}$ is decreasing for the relation $\leq$, that
\[
M^*_{\uF} \leq M^*_{-\nabla f(\xsol(p))} \iff \Jj_{R^*}(M^*_{-\nabla f(\xsol(p))}) \leq \Jj_{R^*}(M^*_{\uF}) .
\]
Using once again Theorem~\ref{thm-FB}, we get
\[
M_{\iter{x}} \leq \Jj_{R^*}(M^*_{-\nabla f(\xsol(p))}) \leq \Jj_{R^*}(M^*_{\uF}),
\]
which yields the second desired inequality.
\end{proof}

%\todo{G: I prefered using below the notation $\xorig$ instead of $\xorig$ because $\xorig$ was used before as the limit of FB, and instead of $\xorig$ because it is used before to represent the data generating the observation $y=\Phi \xorig + w$ and we only impose that $\xorig$ is the solution of $\Pp(0,\yorig)$, not necessarily $\xorig=\xorig$. See if you want to change this.}

%The FB algorithm is popular in imaging sciences and machine learning to solve large scale inverse problems of the form~\eqref{eq-fwd-ip} using~\eqref{eq-ip-primal}, which corresponds to the special case of $f(x)=f(x,p)$ as defined in~\eqref{eq-ip-setup} where $p=(y,\la)$ is a fixed set of parameters.
%
%In this setting, one can put together the results of Theorems~\ref{thm-ip} and~\ref{thm-FB} to obtain a localization results (Proposition~\ref{prop-fb-ip} below) depending only on the data to estimate $\xorig$ assuming that the noise level $\norm{y-\yorig}$ is small enough. 
%
Though the iterates $\iter{x}$ of FB do not converge to $\xorig$, this proposition tells us that the iterates identify an enlarged stratum associated to $\xorig$.
This is an appealing feature from a practical perspective, since one can often make some prior assumption on the sought after vector $\xorig$, such as for instance sparsity or low-rank properties, as we have illustrated in the numerical experiments of Section~\ref{sec-numerics}.

%%%%%%%%%%%%%%%%%%%%%%%%%%%%%%%%%%%%%%%%%%%%%%%%
\subsection{Douglas-Rachford Splitting Algorithm}

%% DR/ADMM %%
The Douglas-Rachford (DR) method \cite{lions1979splitting} is another popular splitting method designed to minimize convex objectives having the additive ``non-smooth + non-smooth'' structure of the form
\begin{equation}\label{eq:pbDR}
	\min_{x\in \RR^N} f(x) + g(x).
\end{equation}
with $f$ and $g$ be proper lsc and convex functions such that $\ri(\dom(f)) \cap \ri(\dom(g)) \neq \emptyset$ and $\Argmin(f+g) \neq \emptyset$.  The DR scheme reads
\begin{equation}\label{eq:DR}
\begin{cases}
\Iter{v} &= \prox_{\ga f}\pa{2\iter{x} - \iter{z}}  , \\
\Iter{z} &= \iter{z} + {\tau_k}\pa{\Iter{v} - \iter{x}}  , \\
\Iter{x} &= \prox_{\ga g} (\Iter{z})  ,
\end{cases}
\end{equation}
where $\gamma > 0$, $\tau_k \in ]0,2[$ is a relaxation parameter. By definition, the DR method is not symmetric with respect to the order of the functions $f$ and/or $g$. Nevertheless, all of our statements throughout hold true, with obvious adaptations, when the order of $f$ and $g$ is reversed in~\eqref{eq:DR}.

It has been shown in \cite{2016-Liang-DR} that under appropriate non-degeneracy assumptions, the DR identifies the active manifolds in finite time, and then \hl{shows} a local linear regime. % (even without quadratic growth) 
These results unify all those that were established in the literature for special problems, see e.g. \cite{DemanetZhang13} for linearly constrained $\ell_1$-minimization, \cite{boley2013local} for quadratic or linear programs, \cite{BauschkeDR14} for feasibility with two subspaces. Under mirror-stratifiability of $f$ or $g$, we get the following \hl{enlarged activity identification} result.

\begin{thm}\label{thm-dr}
Consider the DR iteration \eqref{eq:DR} to solve~\eqref{eq:pbDR} with $\tau_k \in ]0,2[$ such that $\sum_{k\in\NN} \tau_k\pa{2-\tau_k} = +\infty$. Then $\iter{z}$ converges to a fixed point $\zsol$ with $\xsol = \prox_{\ga g}(\zsol) \in \Argmin(f+g)$, and $\iter{x}$ and $\iter{v}$ both converge to $\xsol$. Introducing $\usol = \frac{\zsol - \xsol}{\ga}$, we have furthermore:
\begin{enumerate}[label=(\roman*)] 
\item If $g$ is mirror-stratifiable with respect to the primal-dual stratifications $(\Man^g,\Man^{g^*})$, then for $k$ large enough 
	\[
		\Man^g_{\xsol} \leq \Man^g_{\iter{x}} \leq \Jj_{g^*}(\Man^{g^*}_{\usol}).
	\]
\item If $f$ is mirror-stratifiable with respect to the primal-dual stratifications $(\Man^f,\Man^{f^*})$, then for $k$ large enough 
	\[
		\Man^f_{\xsol} \leq \Man^f_{\iter{v}} \leq \Jj_{f^*}(\Man^{f^*}_{-\usol}) .
	\]
%	\todo{G: I think one can suppose that only $F$ \textit{or} $G$ is stratifiable, not necessary both. Please check. This is maybe important to handle inverse problem. }
%	Assume $F$ and $G$ to be stratified convex functions associated 
%	to the couples of stratifications $(\Man^F,\Man^{F^*})$ and $(\Man^G,\Man^{G^*})$. 
%	We also assume (XXX can we do better?) that there exist a couple
%	primal-dual optimal solutions $(\xorig,\uorig)$ XXX \todo{G: what does ``couple of PD solutions'' means? Should not they rather be the limit of the DR iterations?}
%	Then 
%	
%	\[
%		\Man^{G^*}_{\uorig} \leq \Man^{G^*}_{\iter{u}} \leq \Jj_{G}(\Man^{G}_{\xorig})
%	\]	
%	\todo{G: I have the impression there is notation issue with respect to the previous section, and that $\uorig$ should in fact be $-\uorig$.}
\end{enumerate}
\end{thm}

\begin{proof}
%We write the optimization problem \eqref{eq:pbDR} as the saddle-point problem
%\begin{equation}\label{eq:pbDR2}
%\min_{x\in \RR^N}\max_{u\in \RR^N} F(x) +\dotp{x}{u} + G^*(u)
%\end{equation}
%or the associated variational inequality
%\begin{equation}\label{eq:pbDR3}
%\left(\begin{matrix}0\\0 \end{matrix}\right) \in
%\left(\begin{matrix}\partial F & I \\-I &\partial G^* \end{matrix}\right)\left(\begin{matrix}x\\u \end{matrix}\right)
%\end{equation}
Under the prescribed choice of $\tau_k$, convergence of $\iter{z}$ is ensured by virtue of \cite[Corollary~5.2]{combettes2004solving}. By non-expansiveness of the proximal mapping, and as we are in finite dimension, we also obtain convergence of $\iter{x}$ and $\iter{v}$ to $\xsol$. To prove (i), note that the update of $\iter{x}$ in~\eqref{eq:DR} is equivalent to the monotone inclusion
\begin{equation}\label{eq:dr-g-moninc}
\iter{u} \eqdef \frac{\iter{z} - \iter{x}}{\ga} \in \partial g(\iter{x}) .
\end{equation}
Since $(\iter{x},\iter{u}) \to (\xsol,\usol)$, we conclude about (i) by invoking Theorem~\ref{thm:main}. Similarly, we note that 
the update of $\iter{v}$ in~\eqref{eq:DR} is equivalent to
\begin{equation}\label{eq:dr-f-moninc}
\iter{w} \eqdef \frac{2\iter{x} - \iter{z} - \Iter{v}}{\ga} \in \partial f(\Iter{v}) .
\end{equation}
Using that $(\iter{v},\iter{w}) \to (\xsol,-\usol)$ and applying Theorem~\ref{thm:main} we get (ii).
%optimality conditions satisfying by the prox-operators, we see that the iterates of \eqref{eq:DR} satisfy the following variational inclusion
%\[
%\left\{\quad\begin{matrix}(2\iter{z}-\Iter{x})/\ga &\in& \partial J(\Iter{x}) \\
% 2\iter{x}-\iter{z}-\ga \Iter{u}&\in& \partial G^*(\Iter{u}).\end{matrix}\right.
%\]
%Usual convergence results on DR (XXX write more precisely) show that $\iter{z}$ converges (to $z_0$) and then $\iter{x}$ and $\iter{u}$ as well (modify statement of the theorem ?), so that $z_0=\xorig-\ga \uorig$.
%Since $J$ and $G$ (so $G^*$) are mirror-stratifiable, we can apply Theorem~\ref{thm:main} two times : first with $\iter{x}\to \xorig$ and $2\iter{z}-\Iter{x}\to -\uorig$ and second with 
%$\Iter{u}\to \uorig$ and $2\iter{x}-\iter{z}-\ga \Iter{u}\to \xorig$. This gives the statement of the result.
\end{proof}

%Note : Switching the roles of $F$ and $G$ in the previous development leads to a second DR algorithm and then to a second identification result.

%%%%%%%%%%%%%%%%%%%%%%%%%%%%%%%%%%%%%%%%%%%%%%%%
%\paragraph{Douglas-Rachford for Regularized Inverse Problems}

In the same vein as for FB in the previous section, we now turn to applying the DR scheme \eqref{eq:DR} to solve~\eqref{eq-variational-ip-intro-noisy} by setting in \eqref{eq:pbDR} $f(x)=\frac{1}{2\la} \norm{y - \Phi x}^2$ and $g(x)=R(x)$. 
%We assume again that~\eqref{eq:pbinvcoer} holds. Observe also that in our case, since $\dom(f)=\RR^N$, the domain qualification condition required to apply DR is in force. 
Putting Theorem~\ref{thm-ip} and~\ref{thm-dr}-(i) together, we obtain the following analogue to Proposition~\ref{prop-fb-ip}.

\begin{prop}\label{prop-dr-ip}
Under the assumptions of Theorem~\ref{thm-ip}, 
	%Suppose that $(\Pp(0,\yorig))$ has a unique solution $\xorig$. Assume furthermore that $R$ is mirror-stratifiable with respect to the primal-dual stratifications $(\Strat,\Strat^*)$. 
consider the DR iteration \eqref{eq:DR} to solve~\eqref{eq-variational-ip-intro-noisy} with $\ga > 0$ and $\tau_k \in ]0,2[$ such that $\sum_{k\in\NN} \tau_k\pa{2-\tau_k} = +\infty$.
	%Then there exists $(C_0,C_1)$ depending only on $\xorig$ such that for any $(\la,y)$ obeying~\eqref{eq-cone-domain}, and 
Then, for $k$ large enough, the DR iterates satisfy 
	\eq{
		M_{\xorig} 
		\leq M_{\iter{x}} 
		\leq \Jj_{R^*}(M_{\uF}^*)
		\qwhereq
		\uF \eqdef \Phi^* q(0,\yorig).
	}
%where we recall $q(0,\yorig)$ from~\eqref{eq-min-norm-certif}.
\end{prop}

\begin{proof}
The proof follows the same reasoning as that of Proposition~\ref{prop-fb-ip} by additionally observing (recall the notation in the proof of Theorem~\ref{thm-dr}) that from~\eqref{eq:dr-g-moninc}-\eqref{eq:dr-f-moninc}, we have
$\usol(p)=-\nabla f(\xsol(p)) = \Phi^*\qsol(p) \in \partial R(\xsol(p))$.
\end{proof}

%% file: sections/numerics.tex
% !TEX root = ../StratificationSensitivity.tex

\section{Numerical Illustrations}
\label{sec-numerics}
In this section, we numerically illustrate our theoretical results on sensitivity and \hl{enlarged activity identification} in the context of regularized inverse problems. 
We adopt the same two ``compressed sensing'' scenarios described in Section~\ref{sec-motivnum}.
The dimension $\dim(\Man_{\xorig}) = R_0(\xorig)$ of the strata associated to $\xorig$ is measured as  $R_0=\norm{\cdot}_0$ (resp. $R_0=\rank$) for the $\ell_1$ (resp. nuclear) norm regularization.
%

%%%%
\paragraph{Strata sensitivity}

\hl{We first illustrate the relevance of the strata sensitivity result in Theorem~\ref{thm-ip} by studying the dimension of the largest possible active stratum ${\Jj_{R^*}}(\Man^*_{\uF})$ (in fact its closure).
The dual $\uF=\Phi^* \qF(0,\yorig)$ is computed from $\xorig$ by solving the convex optimization problem~\eqref{eq-dfn-minnorm-cert} (using CVX to get a high precision). Thus we know the maximum complexity index excess predicted by Theorem~\ref{thm-ip}, i.e. 
\eq{
	\de^\star(\xorig) \eqdef \dim( {\Jj_{R^*}}(\Man^*_{\uF}) ) -  \dim(\Man_{\xorig}).
}
For each given $\delta$ and $R_0(\xorig)$, and among the 1000 randomly generated replications of $(\xorig,\Phi)$, we compute the proportion $\rho(R_0(\xorig),\delta)$ of $\xorig$ such that it is the unique solution of $(\Pp(0,\yorig))$ and $\de^\star(\xorig) \leq \de$. The proportions $\rho(R_0(\xorig),\delta)$ are displayed in Figure~\ref{fig-phase-transition} as a function of the input complexity index $R_0(\xorig) =  \dim(\Man_{\xorig})$. The colors from blue to red correspond to increasing $\de$.}

\hl{The proportion $\rho(R_0(\xorig),\delta)$ is an increasing function of $\delta$ and a decreasing function of $R_0(\xorig)$. Indeed, as anticipated from standard compressed sensing results~\cite{2014-vaiter-ps-review}, active strata $M_{\xorig}$ of vectors $\xorig$ whose dimension is small enough compared to the number of measurements $P$ can be provably and stably recovered with overwhelming probability (on the sampling of $(\xorig,\Phi,w)$). As $R_0(\xorig)$ increases, the number of measurements $P$ becomes insufficient to ensure non-degeneracy with high probability, hence preventing stable recovery of $M_{\xorig}$. However, Theorem~\ref{thm-ip} predicts that the active stratum of $\xsol(\la,y)$ for $y$ nearby $\yorig$ is localized between $M_{\xorig}$ and ${\Jj_{R^*}}(\Man^*_{\uF})$.}

\hl{The blue curve in each plot of Figure~\ref{fig-phase-transition} corresponds to $\de=0$, which is the proportion $\rho(R_0(\xorig),0)$ of vectors $\xorig$ whose active stratum $\Man_{\xorig}$ can be recovered stably under small noise perturbation by solving~\eqref{eq-variational-ip-intro-noisy} for $\la$ chosen according to~\eqref{eq-domain-ip}. This proportion shows a phase transition phenomenon between stable recovery and unstable recovery. The location of the phase transition for $\de=0$ can be predicted accurately; see for instance~\cite{2014-vaiter-ps-consistency}.}

\hl{The red curves in Figure~\ref{fig-phase-transition} correspond to the extreme case where $\de$ takes its largest achievable value, i.e.\;where we can guarantee recovery of the largest stratum ${\Jj_{R^*}}(\Man^*_{\uF})$ with high probability. The phase transition occurs for higher dimension $R_0(\xorig)$. The intermediate curves, i.e.\;from blue to red, correspond to the recovered strata that are localized between $M_{\xorig}$ and ${\Jj_{R^*}}(\Man^*_{\uF})$ (i.e.\;increasing $\de$). The phase transition progressively increases with $\de$. These curves illustrate and quantify the typical tradeoff observed in practice: one can allow for more complex input vectors~$\xorig$ (i.e.\,those with larger $R_0(\xorig)$) at the expense of recovering active strata larger than~$\Man_{\xorig}$.}
%

%\todo{it matches the phase transition in the noiseless case (i.e. where one does not look for any stability of the support), see for instance~\cite{amelunxen2013living}.}
%
%The intermediate curves, i.e. from blue to red,
% how the phase transition progressively shifts from full strata stability to no stability at all. This
%illustrate and quantify the typical tradeoff observed in practice: one can allow for more complex input vectors (i.e. larger values of $R_0(\xorig)$) at the expense of a higher unstability in the recovery of the active stratum $\Man_{\xorig}$ in presence of even small noise.

\begin{figure}\centering
\begin{tabular}{@{}c@{\hspace{1mm}}c@{}}
\includegraphics[width=.49\linewidth]{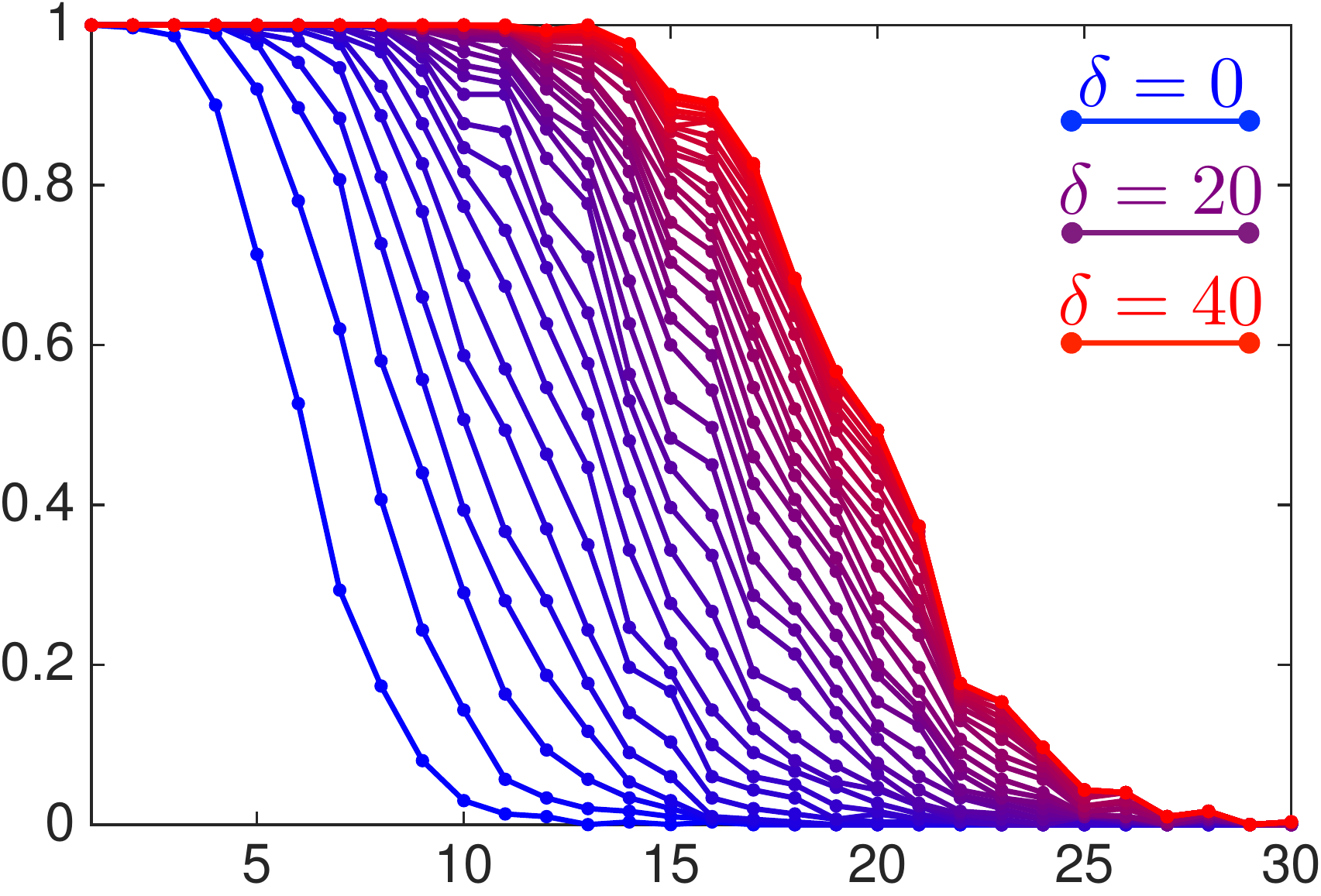} &
\includegraphics[width=.49\linewidth]{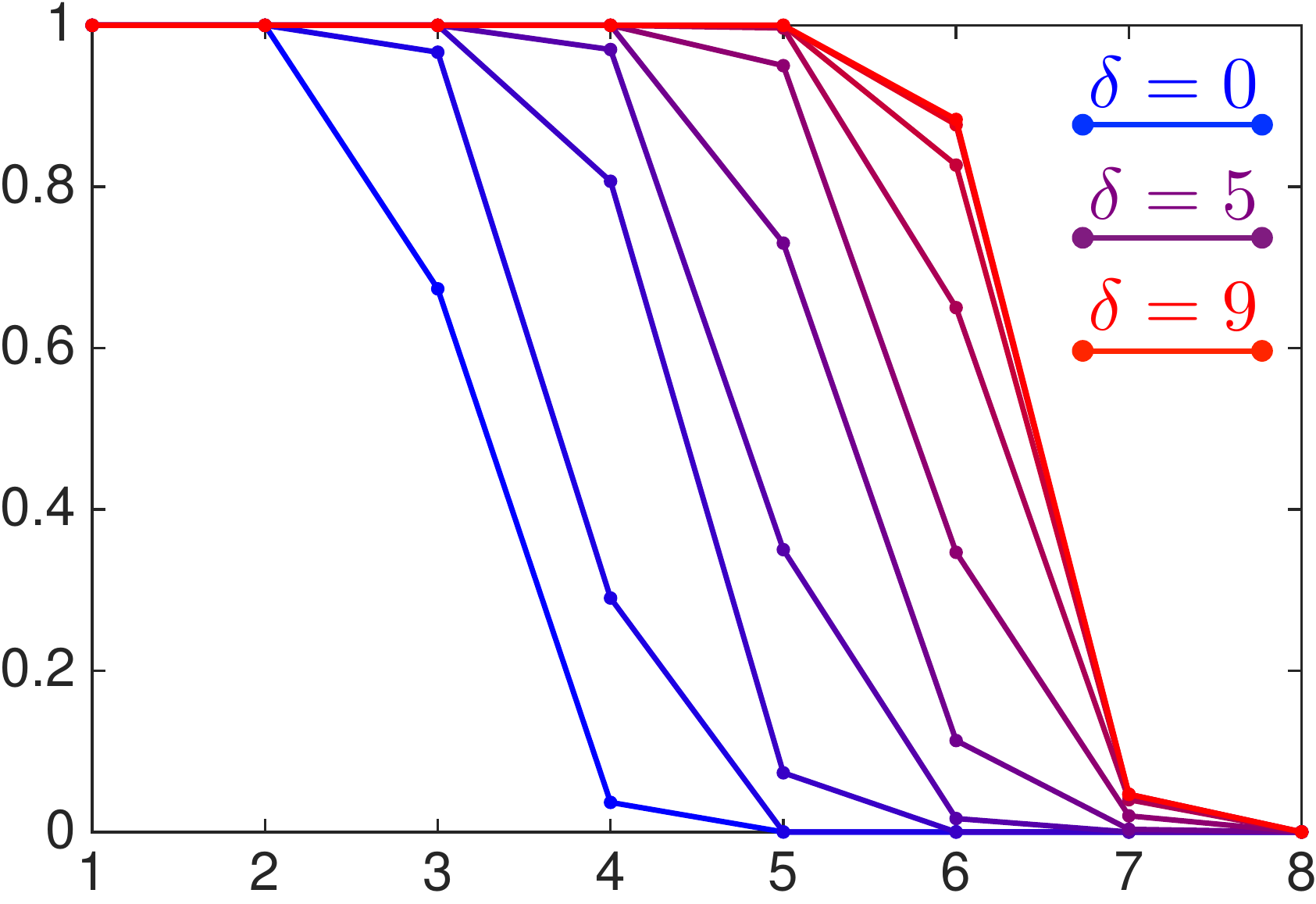}\\
$R=\norm{\cdot}_1$ & $R=\norm{\cdot}_*$
\end{tabular}
\caption{\label{fig-phase-transition}
Proportion of $\xorig$ such that $\de^\star(\xorig) \leq \de$ as a function of $R_0(\xorig)$ (increasing value of $\de$ from $0$ to its maximal value is depicted by a color evolving from blue to red). 
%
%
%Vertical axis is the ratio of random vectors with given $R_0(\xorig)$ such that $\de^\star(\xorig) \leq \de$. 
}
\end{figure}

\begin{figure}\centering
\begin{tabular}{@{}c@{\hspace{1mm}}c@{}}
\includegraphics[width=.49\linewidth]{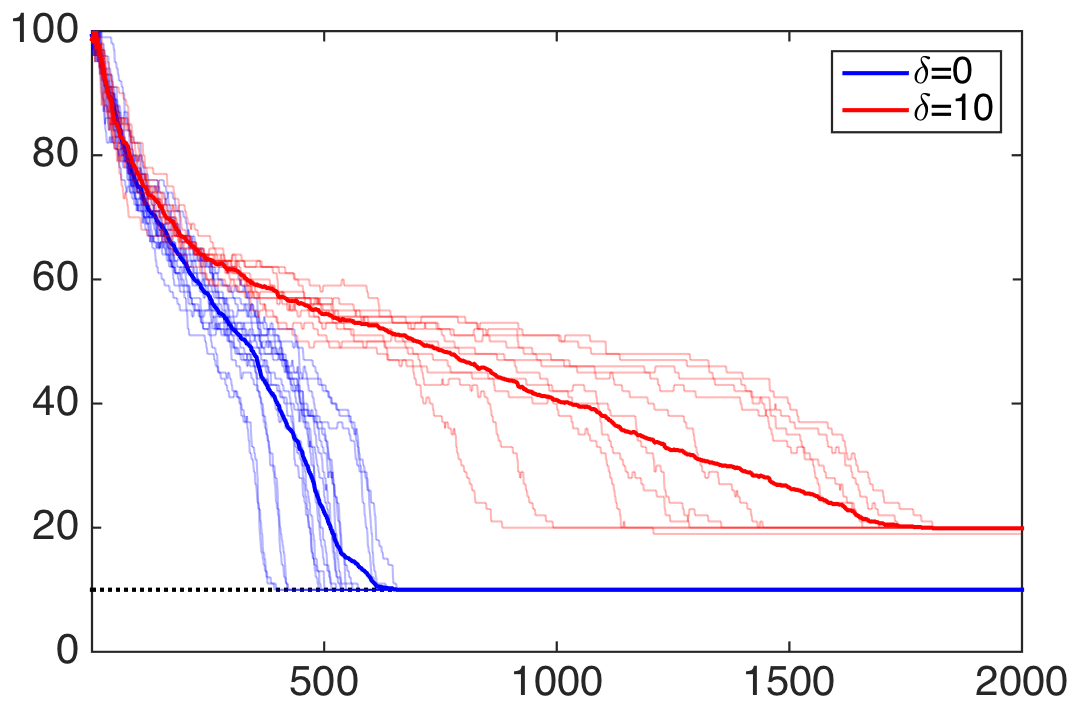} &
\includegraphics[width=.49\linewidth]{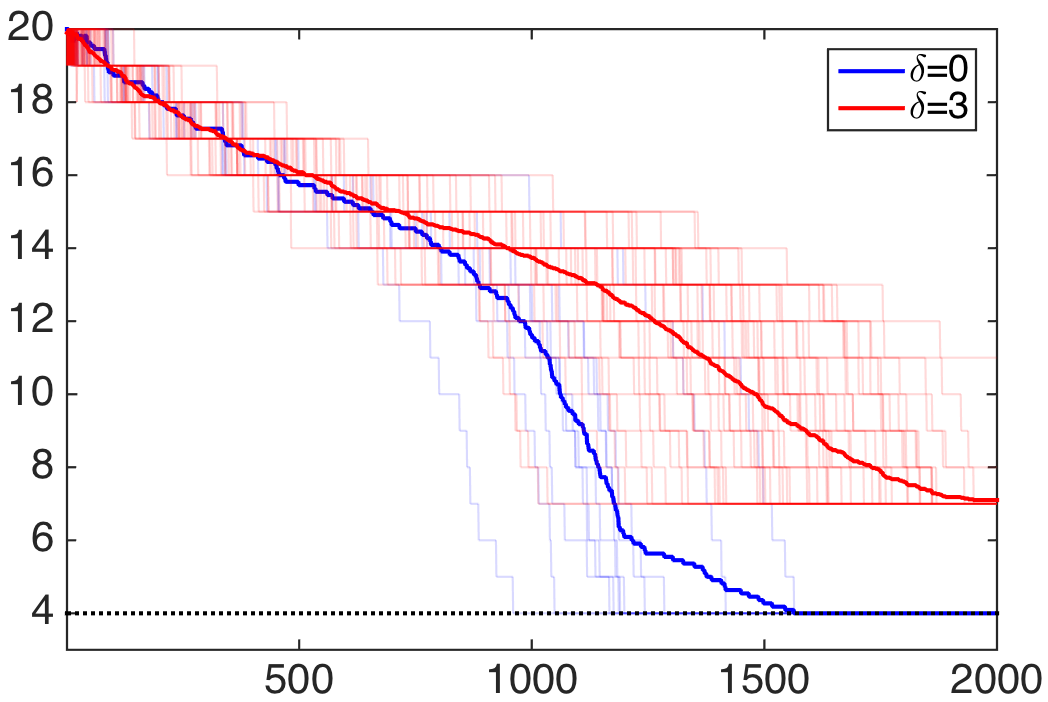}\\
$R=\norm{\cdot}_1$ & $R=\norm{\cdot}_*$
\end{tabular}
\caption{\label{fig-algorithm-paths}
Plots of $R_0(x_k)$ for the 2000 first iterates generated by the Forward-Backward algorithm. Plots in blue correspond to 
the cases when $\de^\star(\xorig)=0$, and the red ones for $\de^\star(\xorig)=\de$.}
\end{figure}

%%%%
\paragraph{Forward-Backward \hl{finite activity localization}}

We now numerically illustrate the \hl{finite enlarged activity identification} of the FB splitting scheme as predicted by Theorem~\ref{thm-FB} and Proposition~\ref{prop-fb-ip}. We remain under the same compressed sensing setting as before. The randomly generated replications of $\xorig$ are such that $R_0(\xorig)=10$ for $R=\norm{\cdot}_1$ and to $4$ for $R=\norm{\cdot}_*$.

The evolution of the complexity index $R_0(\iter{x})$ of the FB iterate $\iter{x}$ is shown in Figure~\ref{fig-algorithm-paths}. 
The blue lines correspond to several trajectories (the bold one is the average trajectory), each for a randomly generated instance of $\xorig$ such that $\de^\star(\xorig)=0$, i.e. those vectors whose active strata $\Man_{\xorig}$ can be exactly recovered under small perturbations. Thus, the iterates identify $\Man_{\xorig}$ in finite time.
The red lines (the bold one is the average trajectory) are those for which $\de^\star(\xorig)=\delta > 0$. As anticipated by our theoretical results, the iterates identify a stratum strictly larger that $\Man_{\xorig}$. 
%while the red trajectories clusters those those (more instable) $\xorig$ for which $\de^\star(\xorig)=\de$ where $\de>0$ is a fixed value (being equal to $10$ for $R=\norm{\cdot}_1$ and to $3$ for $R=\norm{\cdot}_*$).
%The bold curve depicts the average of those clusters, to highlight the different behaviour. 
% 
%While Theorem~\ref{thm-FB} predicts that for $k$ large enough, $R(x_k) \leq R_0(\xorig)+\de$, these numerical illustrations suggest a stronger results. Indeed, this shows that in practice F-B always ultimately selects the largest (maximally unstable) possible stratum, i.e. $R(x_k) = R_0(\xorig)+\de$.

%% file: sections/conclusion.tex
% !TEX root = ../StratificationSensitivity.tex

%\section*{Conclusion}
%In this work, we have presented the class of convex mirror-str\-atifiable functions, whose strong primal-dual structure allows to track active strata. We have provided sensitivity  analysis of parametric composite optimization problems involving such functions, and also studied finite activity identification of proximal splitting algorithms to solve them. In particular, our framework allows to tackle the intricate case where a non-degeneracy (relative interiority) condition, usually assumed in the literature, is not fulfilled. Our work can be extended along several lines. In particular, though we focused here on the FB and DR proximal splitting schemes, our finite identification results can be generalized to cover even more sophisticated splitting algorithms, including for instance primal-dual ones.

%%%%%%%%%%%%%%%%%%%%%%%%%%%%%%%%%
{\small
\section*{Acknowledgements}

The work of Gabriel Peyr\'e has been supported by the European Research Council (ERC project SIGMA-Vision).
Jalal Fadili was partly supported by Institut Universitaire de France.
}

%% file: sections/appendix.tex
% !TEX root = ../StratificationSensitivity.tex

%%%%%%%%%%%%%%%%%%%%%%%%%%%%%%%%%%%%%%%%%%%%%%%%%%%%%%%%%%%%
%%%%%%%%%%%%%%%%%%%%%%%%%%%%%%%%%%%%%%%%%%%%%%%%%%%%%%%%%%%%
\section{Proofs of the results in Section 2}\label{appendix}

%%%%%%%%%%%%%%%%%%%%%%%%%%%%%%%%%%%%%%%%%%%%%%%%%%%%%%%%%%%%
\paragraph{Proof of Proposition~\ref{prop:strat}}
Let us first prove the first equivalence. % by double inclusion. 
Observe that Definition~\ref{eq-active-def} can be read as follows: $\Man$ is active at $x$  if and only if $d(x,\cl(\Man))=0$. Since there is a finite number of strata in the stratification, let us consider $\delta$ the minimum of the nonzero distances $d(x,\cl( \Man'))$ for $\Man'$ not active. For all $x'$ in the open ball of radius $\delta$ and of center $x$, we have
\[
	d(x,\cl(\Man_{x'})) \leq \norm{x-x'}<\delta \qquad \iff \qquad d(x,\cl(\Man_{x'}))= 0 .
\]
This shows that the set of these strata $\Man_{x'}$ indeed coincides the set of active strata, \hl{whence we get the first equivalence}. 

\hl{Let us turn to the second equivalence.} 
Let $\Man$ be an active strata at $x$, that is, $x\in \cl(\Man)$. Since $x\in \Man_x$, the intersection $\cl(\Man)\cap\Man_x$ contains $x$ and thus is nonempty. We deduce from \eqref{eq-order-strata} that $\Man \geq \Man_x$. Conversely, if $\Man \geq \Man_x$, then $x\in \Man_x \subset \cl(\Man)$ and therefore $\Man$ is active at $x$. 
%\qed

%%%%%%%%%%%%%%%%%%%%%%%%%%%%%%%%%%%%%%%%%%%%%%%%%%%%%%%%%%%%
\paragraph{Proof of Proposition~\ref{prop-polyhedral}}

From classical convex analysis calculus rules, we get for all $x\in \dom R$
\eq{%\label{eq-partial-R-poly}
	\partial R(x) = \conv\enscond{a_i}{i\in \Imax(x)} + \cone\enscond{a_i}{i\in \Ifeas(x)} .
}
By definition of $\Man_I$, the relative interior of $\partial R(x)$ is constant over a stratum: for all 
$x\in \Man_I$ (with $\Man_I\neq \emptyset$) 
\begin{eqnarray*}
\ri(\partial R(x)) &=& \ri \big(\conv\enscond{a_i}{i\in \Imax(x)} + \cone\enscond{a_i}{i\in \Ifeas(x)}\big)\\
 &=& \ri (\conv\enscond{a_i }{i\in \Imax(x) }) + \ri(\cone\enscond{a_i}{ i\in \Ifeas(x)})\\
 &=& \ri (\conv\enscond{a_i }{i\in \Imax}) + \ri(\cone\enscond{a_i}{i\in \Ifeas}) .
\end{eqnarray*}
This yields $\Jj_R(\Man_I)= \Man^*_I$. Conversely we have $\partial R^*(u) = \enscond{x}{u\in \partial R(x)}$, which entails for all $u\in M_I^*$
\[
\partial R^*(u) = \enscond{x}{\Imax(x)\cap\Ifeas(x) \supset I} ,
\]
and therefore $\ri(\partial R^*(u)) = M_I$. This gives $\Jj_{R^*}(\Man^*_I)= \Man_I$, which proves item (i) of Definition~\ref{def:mirstrat}.
%\footnote{In fact, we have the expression of $R^*$ but we do not need it... the notion is really about $\partial R$... Using standard calculus rules (see e.g. Theorems 2.3.2 and 2.4.4 of cite[Chap.E]{hull}), we can indeed express the conjugate of $R$
%\begin{eqnarray*}
%R^*(x)&=& 
%\inf\left\{\Big(\max_{i=1,\ldots,k} \{\dotp{a_i}{\cdot} -\al_i \}\Big)^*(q)  
%+ \sigma_{\dom R} (p) :~~ p+q=x\right\}\\
%&=& \inf\left\{\la_1\al_1+\cdots+\la_m\al_m:~~ \la_i\geq 0, \la_1+\cdots+\la_k =1
%\text{ and } \la_1 a_1+\cdots+\la_m a_m= x\right\}
%\end{eqnarray*}}

To show (ii) of Definition~\ref{def:mirstrat}, we make the following observation:
\begin{align*}
\Man_I \leq \Man_{I'}
&\Longleftrightarrow
\text{any $x\in \Man_I$ lies in $\cl(\Man_{I'})$}\\
&\Longleftrightarrow
\Imax=\Imax(x)\supset {(I')}^{\max} \qandq \Ifeas=\Ifeas(x)\supset {(I')}^{\feas}\\
&\Longleftrightarrow
I\supset I' .
\end{align*}
On the other hand we have 
\begin{align*}
\cl(\Jj_R(\Man_{I}))
&\supset \cl(\ri(\conv\enscond{a_i}{i\in \Imax})) + \cl(\ri\cone\enscond{a_i}{i\in \Ifeas}))\\
&=\conv\enscond{a_i}{i\in \Imax} + \cone\enscond{a_i}{i\in \Ifeas}.
\end{align*}
Note that the first $\supset$ is in fact $=$ because $\conv\enscond{a_i}{i\in \Imax}$ is compact.
Using unique decomposition of a polyhedron, we can write, 
\[
\Jj_R(\Man_I) \geq \Jj_R(\Man_{I'}) \Longleftrightarrow \cl(\Jj_R(\Man_I)) \supset \Jj_R(\Man_{I'}) \Longleftrightarrow I\supset I',
\]
which ends the proof.%\qed

%%%%%%%%%%%%%%%%%%%%%%%%%%%%%%%%%%%%%%%%%%%%%%%%%%%%%%%%%%%%
\paragraph{Proof of Proposition~\ref{prop:stratspectral}}

The proof builds upon a key result stated in \cite{daniilidis2013orthogonal}. Theorem~4.6(i) in \cite{daniilidis2013orthogonal} asserts that the collection $\enscond{\si^{-1}(\Man^{\sym})}{\Man \in \Strat}$ forms a smooth stratification of $\dom(R)$ with the desired properties.
The fact that spectral functions are mirror-stratifiable follows from the polyhedral case with the help of Theorem~4.6(iv) in \cite{daniilidis2013orthogonal}, which states that 
\[
\Jj_{R}(\si^{-1}(\Man^{\sym})) = \si^{-1}\big(\Jj_{R^{\sym}}(\Man^{\sym})\big)
\]
together with continuity of the singular value mapping $\si$. %\qed